\documentclass[DIV=16]{scrartcl}
\usepackage[sb]{libertine} %
\usepackage[T1]{fontenc}
\usepackage{textcomp}
\usepackage[varqu,varl]{zi4}%
\usepackage[amsthm,upint]{libertinust1math} %
\usepackage[scr=boondoxo,bb=boondox]{mathalpha} %
\usepackage[pdftex,colorlinks=true,linkcolor={blue!50!black},citecolor={red!50!black},urlcolor=blue]{hyperref}
\usepackage{doi}
\usepackage[numbers]{natbib}
\usepackage{amsmath,amssymb,amsthm}
\usepackage{bm}
\usepackage{graphicx}
\usepackage{booktabs}
\usepackage{hyperref}
\usepackage{placeins}
\newtheorem{theorem}{Theorem}[section]
\newtheorem{problem}[theorem]{Problem}
\newtheorem{lemma}[theorem]{Lemma}
\newtheorem{corollary}[theorem]{Corollary}
\newtheorem{remark}[theorem]{Remark}

\usepackage{amsmath,amssymb,amsfonts,mathtools}
\usepackage{bm}
\usepackage{microtype}
\usepackage{xcolor}
\usepackage{enumitem}
\usepackage{graphicx}
\usepackage{booktabs}
\usepackage{tikz}
\usetikzlibrary{shapes,arrows}
\usetikzlibrary{patterns,decorations.pathreplacing}

\usepackage{algpseudocode}
\usepackage{subcaption}
\usepackage{siunitx}
\usepackage{cleveref}
\crefname{problem}{Problem}{Problems}
\Crefname{problem}{Problem}{Problems}
\crefname{remark}{Remark}{Remarks}
\Crefname{remark}{Remark}{Remarks}
\crefname{assumption}{Assumption}{Assumptions}
\Crefname{assumption}{Assumption}{Assumptions}

\makeatletter
\providecommand*{\toclevel@algorithm}{0}
\providecommand*{\toclevel@assumption}{0}
\providecommand*{\toclevel@theorem}{0}
\providecommand*{\toclevel@lemma}{0}
\providecommand*{\toclevel@proposition}{0}
\providecommand*{\toclevel@corollary}{0}
\providecommand*{\toclevel@remark}{0}
\providecommand*{\toclevel@problem}{0}
\providecommand*{\toclevel@proof}{0}
\makeatother

\setlist[itemize]{leftmargin=2.0em}
\setlist[enumerate]{leftmargin=2.0em}

\DeclareMathOperator{\divv}{div}

\DeclareMathOperator{\sym}{sym}

\DeclareMathOperator{\diag}{diag}

\DeclareMathOperator{\Id}{Id}

\newcommand{\Ip}[2]{\left\langle #1,#2\right\rangle}
\newcommand{\norm}[1]{\left\lVert #1\right\rVert}
\newcommand{\abs}[1]{\left\lvert #1\right\rvert}

\newcommand{\uu}{\boldsymbol v}
\renewcommand{\vv}{\boldsymbol z}

\newcommand{\ff}{\boldsymbol f}
\newcommand{\nn}{\boldsymbol n}

\newcommand{\GG}{\nabla}
\newcommand{\DD}{\boldsymbol D} %
\newcommand{\Sstress}{\boldsymbol S} %

\newcommand{\ljump}{\mathopen{[\mkern-2mu[}}
\newcommand{\rjump}{\mathclose{]\mkern-2mu]}}
\newcommand{\jump}[1]{\ljump #1 \rjump}

\newcommand{\It}{I}
\newcommand{\In}{I_n}
\newcommand{\Sn}{S_n}

\newcommand{\dt}{\partial_t}
\newcommand{\drv}{\mkern3mu\mathrm{d}}
\newcommand{\Dt}{\drv t}

\newcommand{\exact}{\mathrm{exN}}

\newcommand{\clip}{\mathrm{modN}}
\newcommand{\Pic}{\mathrm{Pic}}
\newcommand{\Th}{\mathcal M_h}
\newcommand{\Fh}{\mathcal F_h}
\newcommand{\Fhi}{\mathcal F_h^{\mathrm{int}}}

\newcommand{\GhD}{\Gamma_{D}}
\newcommand{\GhN}{\Gamma_N}

\newcommand{\Vv}{\boldsymbol{V}}
\newcommand{\Vvz}{\boldsymbol{V}_0}
\newcommand{\Pk}{\mathbb P_k}
\newcommand{\Vh}{\boldsymbol V_h}

\newcommand{\pos}{\oplus}
\newcommand{\negpart}{\ominus}
\newcommand{\ypos}[1]{\left(#1\right)^{\pos}}
\newcommand{\yneg}[1]{\left(#1\right)^{\negpart}}

\newcommand{\reg}{\delta}
\newcommand{\pn}{p}
\newcommand{\pnp}{\pn'}

\newcommand{\gOne}{\gamma_{1}}  %
\newcommand{\gTwo}{\gamma_{2}}  %
\newcommand{\hGD}{h_{\Gamma_D}} %
\newcommand{\gCIP}{\gamma_{\mathrm{CIP}}}  %

\newcommand{\eps}{\varepsilon}
\newcommand{\disc}{\mathrm{disc}}

\newcommand{\hF}{h_F}

\newcommand{\etaeff}[1]{\eta \left(#1\right)}

\newcommand{\DStress}[2]{D\Sstress \left(#1\right) \left(#2\right)}

\usepackage{tikz}
\usetikzlibrary{external,spy,shadows}
\tikzset{
  spy using overlaysshadow/.style={
    spy scope={#1,
      every spy on node/.style={
        rectangle,
        fill, fill opacity=0.25, text opacity=1
      },
      every spy in node/.style={
        rectangle, drop shadow,
        fill=white, draw, cap=round
      }
    }
  }
}
\usepackage{pgfplots}
\usepgfplotslibrary{groupplots}
\pgfplotsset{compat=1.18}
\newcommand{\mc}[2]{\multicolumn{1}{#1}{#2}}

\definecolor{myblue}{RGB}{0 83 139}
\definecolor{myred}{RGB}{114 16 69}
\definecolor{mygreen}{RGB}{0 94 0}
\title{}
\date{}
\author{Nils Margenberg\thanks{
    University of Magdeburg,
    Institute for Analysis and Numerics,
    Universit\"atsplatz 2,
    39104 Magdeburg,
    Germany,
    \texttt{\{nils.margenberg,carolin.mehlmann\}@ovgu.de}
  }
  \and Carolin Mehlmann\footnotemark[1]
}
\title{A Scalable Monolithic Modified Newton Multigrid Framework for Time-Dependent \texorpdfstring{$\pn$}{p}-Navier-Stokes Flow}

\begin{document}
\maketitle
\vspace*{-8ex}
\begin{abstract}
  Fully implicit tensor-product space-time discretizations of time-dependent
  $(\pn,\reg)$-Navier-Stokes models yield, on each time step, large nonlinear
  monolithic saddle-point systems. In the shear-thinning regime $1<\pn<2$,
  especially as $\pn\downarrow 1$ and $\reg\downarrow 0$, the decisive
  difficulty is the constitutive tangent: its ill-conditioning impairs Newton
  globalization and the preconditioning of the arising linear systems. We
  therefore develop a scalable monolithic modified Newton framework for
  tensor-product space-time finite elements in which the exact constitutive
  tangent in the Jacobian action is replaced by a better-conditioned surrogate.
  Picard and exact Newton serve as reference linearizations within the same
  algebraic framework. Scalability is achieved through matrix-free operator
  evaluation, a monolithic multigrid V-cycle preconditioner, order-preserving
  reduced Gauss-Radau time quadrature, and an inexact space-time Vanka smoother
  with single-time-point coefficient freezing in local patch matrices. We
  prove coercivity of the linearized viscous-Nitsche term in the uniformly
  elliptic regime $\nu_\infty>0$ and consistency of the reduced time quadrature.
  Numerical tests demonstrate robustness with respect to model parameters,
  nonlinear and linear iteration counts, and scalable parallel performance.
\end{abstract}

\section{Introduction}
\label{sec:intro}
Many geophysical, biomedical, and industrial flows are modeled by shear-thinning
or viscoplastic rheologies; see, e.g.,
\cite{galdi_mathematical_2008,drzisga_stencil_2020,RudiShihStadler2020AdvancedNewton,schussnig_robust_2021,he_linear_2025}.
We consider the time-dependent incompressible $\pn$-Navier-Stokes system with a
regularized $(\pn,\reg)$ stress law
\begin{equation}\label{eq:constitutive-law}
  \Sstress (\DD\uu)
  =\etaeff{\DD\uu}\DD\uu,
  \qquad
  \etaeff{\DD\uu}
  \coloneq \nu_{\infty}+\nu\bigl(\reg^2+\abs{\DD\uu}^2\bigr)^{\frac{\pn-2}{2}},\quad 1<\pn< 2,
\end{equation}
where $\DD\uu\coloneq \tfrac12\bigl(\GG\uu+(\GG\uu)^\top\bigr)$ denotes the
symmetric gradient and $\abs{\DD\uu}^2\coloneq \DD\uu:\DD\uu$ the Frobenius
norm.
We assume $\nu_{\infty}\ge 0$, $\nu>0$, and $\reg>0$.
The regularization $\reg$ removes the singularity at vanishing shear and makes
the stress law $C^1$; if $\nu_\infty>0$, the law is moreover uniformly
elliptic.

To obtain a scalable fully implicit space-time discretization, we employ a
tensor-product DG-in-time approach. This yields, on each time step, a large
nonlinear monolithic saddle-point system that couples all temporal and spatial
degrees of freedom on that step. For such formulations, solver performance is
governed not only by the discretization but also by the nonlinear iteration,
the constitutive linearization, and the efficiency of the
multigrid-preconditioned Krylov solve. This motivates the development of a
scalable monolithic modified Newton framework tailored to this setting.

In the shear-thinning regime, in particular in the limiting case $\pn\downarrow
1$, the exact constitutive tangent contains viscosity-derivative terms that
generate strong anisotropy and poor local conditioning. This can impair both
Newton globalization and Krylov convergence;
cf.~\cite{BarrettLiu1994Quasinorm,BelenkiBerselliDieningRuzicka2012pStokes}.
Picard iterations avoid these terms, but typically at the price of slower local
convergence; see, e.g., \cite{Toulopoulos2023DGPowerLaw}. The central issue is
therefore the construction of a scalable, robust, monolithic space-time solver
for strongly nonlinear shear-thinning $(\pn,\reg)$-Navier-Stokes flow.

To achieve mesh and parameter robustness in this regime, we employ a modified
Newton solver that replaces the exact constitutive tangent in the Jacobian
action by a better-conditioned surrogate while leaving the nonlinear residual
unchanged. Picard and exact Newton serve as reference linearizations within
the same discrete residual equation. This follows the general strategy of replacing ill-conditioned exact tangents
by better-conditioned linearizations while keeping the nonlinear residual
unchanged;
see~\cite{RudiShihStadler2020AdvancedNewton,ShihMehlmannLoschStadler2023PrimalDualNewton}.
Further improvements of the globalization strategy through the line search are
discussed in~\cite{schmidtNiko2025,wickTtoul2017}.

To achieve scalability of the monolithic solver, the nonlinear iteration is
combined with a monolithic space-time multigrid preconditioner, matrix-free
time-step operator evaluation, weak imposition of Dirichlet data by Nitsche
terms, and convection-aligned CIP stabilization. The multigrid hierarchy
exploits the tensor-product structure in time and space, while the smoother
uses local Vanka patch solves. Because fully time-dependent patch assembly is
expensive, we fix the state-dependent coefficients at one representative time
point per time step when assembling the local patch matrices. This is the only
approximation inside the preconditioner; cf.~\cite{MargenbergBause2026MonolithicSTMG}.

For the fully discretized system in the uniformly elliptic regime $\nu_\infty>0$,
we prove coercivity of the linearized viscous-Nitsche term, derive a discrete
energy estimate for the convection-free problem, and justify the reduced
Gauss-Radau time quadrature used in the implementation. The numerically most
challenging regime is, however, the nearly degenerate limit
$\nu_\infty\downarrow 0$ and $\pn\downarrow 1$. We therefore also include
$\nu_\infty=0$ in the numerical experiments to assess solver robustness.

\paragraph{Contributions and scope}
We present a scalable mesh-robust multigrid-preconditioned Newton framework for
tensor-product space-time finite element discretizations of the time-dependent
$(\pn,\reg)$-Navier-Stokes equations.
\begin{enumerate}
\item \textbf{Conditioning of the Jacobian.} We formulate Picard, exact
  Newton, and modified Newton methods for the linearization of the monolithic
  space-time system. In the modified Newton method, the exact constitutive
  tangent in the Jacobian action is replaced by a better-conditioned
  surrogate, leading to an improved conditioning of the Jacobian
  (\Cref{sec:newton,sec:newton-krylov}).

\item \textbf{Scalable monolithic algebraic realization.} We combine these
  nonlinear iterations with matrix-free time-marching operators and a monolithic
  space-time multigrid preconditioner with a Vanka-type smoother. Therein, a
  surrogate patch assembly based on single-time-point evaluations is used
  (\Cref{sec:algebraic,sec:newton-krylov}).

\item \textbf{Computational assessment in the shear-thinning regime.} The
  numerical study demonstrates that mesh robustness can only be achieved with the
  modified Newton solver in the shear-thinning regime for $\pn\downarrow 1$ and
  $\reg \downarrow 0$ (\Cref{sec:numerics}).
\end{enumerate}

\subsection{Related work}
Foundational PDE theory and finite element analysis in the natural-distance
framework for incompressible flows with $\pn$-structure is
provided in \cite{MalekNecasRokytaRuzicka1996,DieningRuzickaWolf2010Weak}
and \cite{BarrettLiu1994Quasinorm,BelenkiBerselliDieningRuzicka2012pStokes}, respectively.
Implicit time discretizations and fully discrete convergence results for
shear-dependent viscosities are discussed for example in
\cite{ProhlRuzicka2001FullyImplicit,BerselliKaltenbach2025Convergence}. Recent discretization developments for unsteady $\pn$-Navier-Stokes problems
include local discontinuous Galerkin methods
\cite{KaltenbachRuzicka2023LDGpNS,KaltenbachRuzicka2023PartII,KaltenbachRuzicka2024PartIII},
virtual element schemes for pseudoplastic and non-Newtonian Stokes flows on
polygonal meshes \cite{AntoniettiEtAl2024VEMNonNewtonian}, and stabilized
time-DG space-time discretizations for power-law-type models
\cite{Toulopoulos2023DGPowerLaw}.

Large monolithic linear systems arising from generalized Stokes and
Oseen-type problems, monolithic multigrid and block-preconditioned Krylov
methods are well developed~\cite{drzisga_stencil_2020,schussnig_robust_2021,kohlTextbookEfficiencyMassively2022,abu-labdeh_monolithic_2023,jodlbauerMatrixfreeMonolithicMultigrid2024,prietosaavedraMatrixFreeStabilizedSolver2024,rafiei2025improvedmonolithicmultigridmethods,voroninMonolithicMultigridPreconditioners2024}. Recent developments include matrix-free high-order realizations and
$hp$-robust variants with Vanka-type patch smoothing; see, e.g.,
\cite{MargenbergMunchBause2025hpMGStokes,MargenbergBause2026MonolithicSTMG}.
On the nonlinear side, modified Newton-Krylov strategies based on replacing the
exact tangent by better-conditioned surrogates have proved effective for
viscous-plastic Stokes and viscous-plastic sea-ice models
\cite{RudiShihStadler2020AdvancedNewton,ShihMehlmannLoschStadler2023PrimalDualNewton}.

Our goal is to develop a modified Newton-Krylov method for fully implicit
monolithic space-time discretizations of shear-thinning
$(\pn,\reg)$-Navier-Stokes flow. To the best of our knowledge, the
combination of shear-thinning $(\pn,\reg)$-Navier-Stokes flow, a fully
implicit monolithic space-time discretization, and modified Newton-Krylov
iterations with monolithic space-time multigrid preconditioning has not been
studied before. Existing works on $\pn$-Navier-Stokes discretization mainly
address stability and convergence of the discretization~\cite{KaltenbachRuzicka2023LDGpNS,KaltenbachRuzicka2023PartII,KaltenbachRuzicka2024PartIII,Toulopoulos2023DGPowerLaw},
whereas our focus is on scalability, mesh-robust solver performance, and the
treatment of the constitutive tangent in a modified Newton method.

\paragraph{Organization}
\Cref{sec:model} states the governing equations and weak formulation.
\Cref{sec:discretization} introduces the fully implicit tensor-product
space-time DG formulation with Nitsche and CIP terms.
\Cref{sec:newton} outlines the modified Newton solver.
\Cref{sec:algebraic} and \Cref{sec:newton-krylov} describe the time-step-wise
algebraic realization, matrix-free operators, and the monolithic multigrid
preconditioned Krylov solve.
Numerical experiments are reported in \Cref{sec:numerics}, and
\Cref{sec:conclusions} concludes the paper.

\section{Time-dependent \texorpdfstring{$(\pn,\reg)$}{(p,delta)}-Navier-Stokes system}\label{sec:model}
We state the strong and weak formulations and fix the notation used in the
discretization and linearization sections.

\subsection{Problem statement and weak formulation}\label{sec:model:problem}
Let $\Omega\subset\mathbb{R}^d$, $d\in\{2,3\}$, be a bounded Lipschitz domain,
let $\It=(0,T)$, and assume a boundary decomposition
$\partial\Omega=\GhD\cup\GhN$ with $\GhD\cap\GhN=\emptyset$ and outer unit
normal $\nn$.
We write $\Gamma\coloneq\partial\Omega$.
Given are a body force $\ff$, Dirichlet data $\boldsymbol g_D$, and an initial
velocity $\uu_0$.

\paragraph{Spaces and pairings}
Fix $\pn\in(1,\infty)$ and set $\pnp=\pn/(\pn-1)$.
We use the same symbol $\Ip{\cdot}{\cdot}$ for duality pairings and $L^2$
inner products; when both arguments lie in $L^2$, the two coincide.
Define
\[
  \Vvz \coloneq \{ \uu\in \boldsymbol H^1(\Omega) : \uu=\boldsymbol 0 \text{ on }\GhD\},
\]
and $Q\coloneq L^2(\Omega)$. If $\GhN=\emptyset$, the pressure is determined
only up to an additive constant, and one may equivalently work in
$L^2_0(\Omega)$. For time-dependent Dirichlet data we interpret, for a.e.\
$t\in\It$,
\[
  \Vv(t)\coloneq \boldsymbol g_D(t)+\Vvz,
\]
so the weak formulation is posed in the affine space $\Vv(t)\times Q$ for
a.e.\ $t\in\It$.

\paragraph{Constitutive law}
We use the regularized $(\pn,\reg)$ constitutive law \eqref{eq:constitutive-law}
and restrict the discussion to the shear-thinning regime $1<\pn < 2$.
For brevity we often write $\bm A=\DD\uu$.

\paragraph{Strong form}
Find $(\uu,\pi)$ such that
\begin{equation}
\label{eq:pns-strong}
\begin{aligned}
  \dt \uu
  +\divv(\uu\otimes\uu)
  -\divv \Sstress(\DD\uu)
  +\GG \pi &= \ff &&\text{in }\Omega\times \It,\\
  \divv \uu &= 0 &&\text{in }\Omega\times \It,\\
  \uu &= \boldsymbol g_D &&\text{on }\GhD\times \It,\\
  \bigl(\Sstress(\DD\uu)-\pi\Id\bigr)\nn &= \boldsymbol 0 &&\text{on }\GhN\times \It,\\
  \uu(\cdot,0)&=\uu_0 &&\text{in }\Omega.
\end{aligned}
\end{equation}

\paragraph{Weak form}
Testing \eqref{eq:pns-strong} with $(\vv,q)\in \Vvz\times Q$, integrating the
viscous and pressure terms by parts, and using $\vv|_{\GhD}=0$, we obtain:
find $(\uu(t),\pi(t))\in \Vv(t)\times Q$ such that, for a.e.\ $t\in \It$,
\begin{equation}
\label{eq:pns-weak}
\begin{aligned}
  \Ip{\dt\uu}{\vv}
  &+\Ip{\Sstress(\DD\uu)}{\DD\vv}
  -\Ip{\uu\otimes\uu}{\GG\vv}
  +\Ip{(\uu\cdot\nn)\uu}{\vv}_{\GhN}
\\
  &-\Ip{\divv\vv}{\pi}
  -\Ip{\divv\uu}{q}
  =\Ip{\ff}{\vv}
  \qquad \forall(\vv,q)\in \Vvz\times Q,
\end{aligned}
\end{equation}
together with $\uu(\cdot,0)=\uu_0$ in $\boldsymbol L^2(\Omega)$.
Thus the convective boundary term acts only on
$\GhN$.
For later reference we write \eqref{eq:pns-weak} as
\[
  \Ip{\dt\uu}{\vv}
  +a(\uu)(\vv)
  +c(\uu)(\vv)
  +b(\vv,\pi)+b(\uu,q)
  =\Ip{\ff}{\vv},
\]
with
\[
  a(\uu)(\vv)\coloneq \Ip{\Sstress(\DD\uu)}{\DD\vv},
  \qquad
  b(\vv,q)\coloneq -\Ip{\divv\vv}{q},
\]
and
\[
  c(\uu)(\vv)\coloneq -\Ip{\uu\otimes\uu}{\GG\vv}
  +\Ip{(\uu\cdot\nn)\uu}{\vv}_{\GhN}.
\]

\paragraph{Regularity and well-posedness}
Existence of weak solutions for generalized Newtonian fluids with
$(\pn,\reg)$-structure, including the case $\nu_\infty=0$, is covered by
\cite{DieningRuzickaWolf2010Weak}.
For variable exponent models $p(\bm x,t)$ and fully discrete implicit schemes,
see \cite{BerselliKaltenbach2025Convergence}.
In the present paper we exploit that, for $\nu_\infty>0$ and $\reg>0$, the
stress law \eqref{eq:constitutive-law} is strictly monotone, uniformly elliptic,
and $C^1$ with respect to the symmetric gradient; see
\cite{BarrettLiu1993Carreau,BarrettLiu1994Quasinorm}.
For $\nu_\infty=0$ and $\reg>0$, strict monotonicity and $C^1$-regularity
remain valid, but uniform ellipticity is lost for large shear.
Accordingly, the analysis below is carried out only for $\nu_\infty>0$, whereas
the case $\nu_\infty=0$ is treated numerically. In particular, the
$\boldsymbol H^1$-based weak formulation adopted above matches the uniformly
elliptic regime addressed analytically in the remainder of the paper. In the
implementation we nevertheless use the same $\boldsymbol H^1$-conforming
discrete velocity space also for $\nu_\infty=0$; in that regime the natural
coercivity scale is $\boldsymbol W^{1,\pn}$.

\section{Space-time discretization}
\label{sec:discretization}

We use a DG discretization in time and an inf-sup stable conforming finite
element discretization in space.
The spaces are introduced in \Cref{sec:st-spaces}, the spatial semilinear forms
in \Cref{sec:st-forms}, and the fully discrete problem in
\Cref{sec:st-fully-discrete}.

\subsection{Spaces and meshes}
\label{sec:st-spaces}

Let $0=t_0<t_1<\dots<t_N=T$ and denote the temporal mesh by
$\mathcal M_\tau\coloneq\{\In\}_{n=1}^N$ with
\begin{equation}
  \label{eq:In-Sn}
  \In=(t_{n-1},t_n],\quad\text{and}\quad\Sn=\Omega\times \In.
\end{equation}
We refer to \(\In\) as a time step and to \(\Sn\) as the corresponding
space-time slab.
Set $\tau_n\coloneq t_n-t_{n-1}$ and $\tau\coloneq\max_n\tau_n$.
For $k\in\mathbb N_0$ define the broken polynomial space
\begin{equation}
\label{eq:DefYtau}
  Y_\tau^k(\It)\coloneq
  \{w_\tau:\It\to\mathbb R \mid (w_\tau)_{|\In}\in\Pk(\In;\mathbb R)\ \forall \In\in\mathcal M_\tau\}.
\end{equation}

Let $\Th$ be a shape-regular mesh of $\Omega$ with element diameter $h_T$ and
$h\coloneq\max_{T\in\Th}h_T$.
We denote by $\Fhi$ the set of interior faces, by $\Fh^\Gamma$ the set of
boundary faces, and by $\hF$ a local face diameter.
For $F\in\Fhi$ we fix an orientation by a unit normal $\nn_F$ and use the
standard jump notation $\jump{\cdot}$ across $F$.

Let $(\Vh,Q_h)$ be an inf-sup stable finite element pair on $\Omega$, with
$\Vh\subset \boldsymbol H^1(\Omega)$ (for example Taylor-Hood, or
Scott-Vogelius on suitable meshes).
We assume that the discrete inf-sup constant and the trace and Korn constants
remain uniform under refinement, up to the standard dependence on the polynomial
degree.
Set $\boldsymbol X_h\coloneq \Vh\times Q_h$ and define the discretely
divergence-free subspace
\[
  \Vh^{\operatorname{div}}\coloneq
  \{\uu_h\in\Vh\mid \Ip{\divv\uu_h}{q_h}=0\ \forall q_h\in Q_h\}.
\]

The space-time trial and test space is
\[
  \boldsymbol X_{\tau h}^{k}
  \coloneq \big(Y_\tau^k(\It)\otimes \Vh\big)\times \big(Y_\tau^k(\It)\otimes Q_h\big).
\]
For $w_\tau\in Y_\tau^k(\It)$ we denote by $w_\tau^-(t_n)$ and $w_\tau^+(t_n)$
the left and right traces at $t_n$, and by
$\jump{w_\tau}_n\coloneq w_\tau^+(t_n)-w_\tau^-(t_n)$ the temporal jump.
If an initial value $w_0$ is prescribed, we set $w_\tau^-(t_0)\coloneq w_0$ and
$\jump{w_\tau}_0=w_\tau^+(t_0)-w_0$.
In particular, we prescribe $\uu_{\tau h}^-(t_0)\coloneq \uu_{0,h}$.

\paragraph{Gauss-Radau temporal nodes}
On each interval $\In$ we fix the right-sided $(k{+}1)$-point Gauss-Radau nodes
$\{t_n^\mu\}_{\mu=1}^{k+1}$ and use the associated Lagrange basis
$\{\ell_\mu\}_{\mu=1}^{k+1}\subset\Pk(\In)$ satisfying
$\ell_\mu(t_n^\nu)=\delta_{\mu\nu}$.
Thus the temporal degrees of freedom on $\In$ are nodal values at
$\{t_n^\mu\}$.
Because $t_n$ is one of the Gauss-Radau nodes, $w_\tau(t_n^-)$ is a nodal
value on $\In$, whereas $w_\tau(t_{n-1}^+)$ is obtained by evaluating the
Lagrange expansion at the left endpoint.

\subsection{Semilinear spatial form and stabilization}
\label{sec:st-forms}

We first define the semilinear spatial form and then specify its stabilization
and boundary terms.

\paragraph{Semilinear forms}
For $\boldsymbol u_h=(\uu_h,\pi_h)\in\boldsymbol X_h$ and
$\boldsymbol w_h=(\vv_h,q_h)\in\boldsymbol X_h$, set
\begin{equation}
\label{eq:pns-form}
  A(\boldsymbol u_h)(\boldsymbol w_h)
  \coloneq
  a(\uu_h)(\vv_h)
  +c(\uu_h)(\vv_h)
  +b(\vv_h,\pi_h)+b(\uu_h,q_h),
\end{equation}
and define
\begin{equation}
\label{eq:def-a-gam}
\begin{split}
A_\gamma(\boldsymbol u_h)(\boldsymbol w_h)
&\coloneq
A(\boldsymbol u_h)(\boldsymbol w_h)
-\Ip{(\Sstress(\DD\uu_h)-\pi_h\Id)\nn}{\vv_h}_{\GhD}
\\
&\qquad
+ B_\gamma(\uu_h,\boldsymbol w_h)
+ s_{\mathrm{CIP}}(\uu_h;\uu_h,\vv_h).
\end{split}
\end{equation}
We use the divergence form of the convective term:
for $\uu_h,\vv_h\in\Vh$,
\begin{equation}
\label{eq:Defc}
  c(\uu_h)(\vv_h)
  \coloneq
  -\Ip{\uu_h\otimes \uu_h}{\GG \vv_h}
  +\Ip{(\uu_h\cdot\nn)\uu_h}{\vv_h}_{\Gamma}.
\end{equation}
Since Dirichlet data are imposed weakly, $\Vh$ does not vanish on $\GhD$.
Accordingly, the boundary integral in \eqref{eq:Defc} is taken over the full
boundary, and the inflow correction on $\GhD$ is handled by $B^c$ in
\eqref{eq:def-b-gam}.

\paragraph{CIP stabilization}
On the interior skeleton $\Fhi$ we add a convection-aligned CIP term.
For $\hat{\uu}\in\Vh$ and $\uu_h,\vv_h\in\Vh$ we set
\begin{equation}
\label{eq:scip-def}
  s_{\mathrm{CIP}}(\hat{\uu};\uu_h,\vv_h)
  \coloneq \sum_{F\in\Fhi} \gCIP\,\hF^2\,
  \Ip{\abs{\hat{\uu}\cdot\nn_F}\jump{\GG \uu_h\cdot\nn_F}}
     {\jump{\GG \vv_h\cdot\nn_F}}_F.
\end{equation}
This term acts only on convective transport and does not modify the
velocity-pressure coupling.
Since $(\Vh,Q_h)$ is inf-sup stable, no pressure stabilization is added.

\paragraph{Boundary terms: Nitsche with fixed viscous coefficients}
Dirichlet data are imposed weakly by Nitsche terms;
see \cite{Nitsche1971Dirichlet,JuntunenStenberg2009Nitsche}.
For non-Newtonian models, related weak-imposition ideas are discussed in
\cite{BaltussenEtAl2011WeakDirichletNonNewtonian}.

For $y\in\mathbb R$ write
$\ypos{y}\coloneq \tfrac12(\abs{y}+y)$ and
$\yneg{y}\coloneq \tfrac12(\abs{y}-y)$.
Let $\gOne>0$ and $\gTwo>0$ be penalty parameters and let $\hGD$ denote a local
boundary mesh size on $\GhD$.
Further, let $\etaeff{\cdot}$ be the effective viscosity from
\eqref{eq:constitutive-law}.

Let $\boldsymbol g_D$ denote the prescribed Dirichlet trace on $\GhD$ and fix a
lifting $\hat{\boldsymbol g}_D$ to $\Omega$.
Whenever $B_\gamma(\boldsymbol g_D,\cdot)$ appears below, it is understood
through this fixed lifting.
In the viscous Nitsche terms we evaluate both the effective viscosity and the
stress derivative at $\DD\hat{\boldsymbol g}_D$.
This keeps the boundary contribution bilinear, while the volume term
$\uu\mapsto \Sstress(\DD\uu)$ remains fully nonlinear.
Accordingly, define
$B_\gamma(\cdot,\cdot):\boldsymbol H^1(\Omega)\times \boldsymbol X_h\to\mathbb R$
by
\begin{subequations}
\label{eq:def-b-gam}
\begin{alignat}{2}
\label{eq:def-b-gam-1}
B_\gamma(\uu,\boldsymbol w_h)
&\coloneq B^{c}(\uu,\boldsymbol w_h)+B^{s}(\uu,\boldsymbol w_h)+B_\gamma^{r}(\uu,\boldsymbol w_h),
\\
\label{eq:def-b-gam-2}
B^{c}(\uu,\boldsymbol w_h)
&\coloneq -\Ip{\yneg{\uu\cdot\nn}\,\uu}{\vv_h}_{\GhD},
\\
\label{eq:def-b-gam-2.5}
B^{s}(\uu,\boldsymbol w_h)
&\coloneq -\Ip{\uu}{\Big(\DStress{\DD\hat{\boldsymbol g}_D}{\DD\vv_h}+q_h\Id\Big)\nn}_{\GhD},
\\
\label{eq:def-b-gam-3}
B_\gamma^{r}(\uu,\boldsymbol w_h)
&\coloneq \gOne\,\hGD^{-1}\Ip{\etaeff{\DD\hat{\boldsymbol g}_D}\,\uu}{\vv_h}_{\GhD}
+\gTwo\,\hGD^{-1}\Ip{\uu\cdot\nn}{\vv_h\cdot\nn}_{\GhD}.
\end{alignat}
\end{subequations}

\paragraph{Consistency and adjoint consistency}
If $\uu\in\boldsymbol H^1(\Omega)$ satisfies
$\uu|_{\GhD}=\boldsymbol g_D$ in the trace sense, then
\(B_\gamma(\uu,\boldsymbol w_h)=B_\gamma(\boldsymbol g_D,\boldsymbol w_h)\) for
all $\boldsymbol w_h\in\boldsymbol X_h$.
Indeed, all terms in \eqref{eq:def-b-gam} depend on $\uu$ only through its
trace on $\GhD$.
Fixing the viscous coefficients at $\DD\hat{\boldsymbol g}_D$ generally
destroys adjoint consistency.
Accordingly, we do not pursue duality-based $L^2$-error estimates here.
For stability, $\gOne$ (and $\gTwo$ for the normal component) is chosen
sufficiently large so that the standard Nitsche trace argument applies;
cf.~\cite{JuntunenStenberg2009Nitsche}.

\subsection{Fully discrete space-time variational problem and energy stability}
\label{sec:st-fully-discrete}

\begin{problem}[Discrete space-time variational problem]
\label{Prob:DVP-pns}
Assume
\[
  \ff\in L^{\pnp}(\It;\boldsymbol W^{-1,\pnp}(\Omega)),
  \qquad
  \boldsymbol g_D\in L^2(\It;\boldsymbol H^{1/2}(\GhD)),
\]
Let $\uu_{0,h}\in \Vh$ be an approximation of the initial velocity $\uu_0$ and
prescribe $\uu_{\tau h}^-(t_0)\coloneq \uu_{0,h}$.
Find $\boldsymbol u_{\tau h}=(\uu_{\tau h},\pi_{\tau h})\in \boldsymbol X_{\tau h}^{k}$
such that for all $\boldsymbol w_{\tau h}=(\vv_{\tau h},q_{\tau h})\in \boldsymbol X_{\tau h}^{k}$,
\begin{equation}
\label{eq:disc}
\begin{aligned}
\sum_{n=1}^{N}\int_{t_{n-1}}^{t_n} &\Ip{\dt \uu_{\tau h}}{\vv_{\tau h}}
+ A_\gamma(\boldsymbol u_{\tau h}(t))(\boldsymbol w_{\tau h}(t))\Dt
+ \sum_{n=1}^{N}\Ip{\jump{\uu_{\tau h}}_{n-1}}{\vv_{\tau h}^+(t_{n-1})}
\\
&=
\sum_{n=1}^{N}\int_{t_{n-1}}^{t_n}\Ip{\ff}{\vv_{\tau h}}\Dt
+ \sum_{n=1}^{N}\int_{t_{n-1}}^{t_n} B_\gamma(\boldsymbol g_D(t),\boldsymbol w_{\tau h}(t))\Dt.
\end{aligned}
\end{equation}
\end{problem}

We adopt the residual convention \emph{right-hand side minus left-hand side}.
For $\bm U=(\uu,\pi)\in \boldsymbol X_{\tau h}^k$ and
$\bm W=(\vv,q)\in \boldsymbol X_{\tau h}^k$, define the discrete residual
functional
\begin{equation}
\label{eq:disc-residual}
\begin{aligned}
\mathcal R_{\tau h}(\bm U;\bm W)
\coloneq {}&
\sum_{n=1}^{N}\int_{t_{n-1}}^{t_n}\Ip{\ff}{\vv}\,\Dt
+ \sum_{n=1}^{N}\int_{t_{n-1}}^{t_n} B_\gamma(\boldsymbol g_D(t),\bm W(t))\,\Dt
\\
&-
\sum_{n=1}^{N}\int_{t_{n-1}}^{t_n}\Ip{\dt \uu}{\vv}\,\Dt
- \sum_{n=1}^{N}\int_{t_{n-1}}^{t_n} A_\gamma(\bm U(t))(\bm W(t))\,\Dt
\\
&- \sum_{n=1}^{N}\Ip{\jump{\uu}_{n-1}}{\vv^+(t_{n-1})}.
\end{aligned}
\end{equation}
Then $\boldsymbol u_{\tau h}\in \boldsymbol X_{\tau h}^k$ solves \eqref{eq:disc} if and
only if
\[
  \mathcal R_{\tau h}(\boldsymbol u_{\tau h};\bm w_{\tau h})=0
  \qquad
  \forall\,\bm w_{\tau h}\in \boldsymbol X_{\tau h}^k.
\]

\begin{remark}[Global space-time form and time stepping]
\label{rem:time-step-wise}
Problem~\ref{Prob:DVP-pns} is written in monolithic global space-time form.
Because $\boldsymbol X_{\tau h}^k$ is discontinuous in time, the residual on a
time step $\In$ depends only on $\boldsymbol u_{\tau h}|_{\In}$ and on the left
trace $\uu_{\tau h}^-(t_{n-1})$ through the DG jump term. After fixing a basis
in $Y_\tau^k$ and ordering unknowns in temporal order, the global nonlinear
system is therefore causal and its Jacobian is block lower bidiagonal in time.
In the implementation we therefore realize the method by time marching: on each
time step $\In$ we solve one monolithic space-time problem for all temporal and
spatial degrees of freedom on $\In$, with the left trace from $t_{n-1}$ treated
as data. The multigrid preconditioner acts on this time-marching system, not on
the global space-time system at once.
\end{remark}

\begin{remark}[Gauss-Radau quadrature]\label{rem:gauss-radau}
In the implementation, time integrals involving state-dependent coefficients,
notably the $\pn$-stress and convection terms, are evaluated on each $\In$ by
the right-sided $(k{+}1)$-point Gauss-Radau rule.
The rule preserves the endpoint $t_n$ needed by the DG jump term and is exact
for polynomial integrands up to degree $2k$.
For the present nonlinear constitutive law this is a variational crime; its
effect is analyzed in Appendix~\ref{sec:gr-underint}.
\end{remark}

\section{Nonlinear solution strategies in each time step}
\label{sec:newton}

Each time-step problem produced by the fully implicit space-time discretization
is a large nonlinear saddle-point system. All nonlinear variants considered
below solve one fixed discrete residual equation and differ only in the
linearization used in the correction equation. The central method is modified
Newton; Picard and exact Newton are included only as reference linearizations
within the same discretization and the same algebraic framework.

The constitutive term is the decisive difficulty.
For $1<\pn < 2$, the exact constitutive derivative is anisotropic and its
smallest eigenvalue deteriorates as $\pn\downarrow 1$.
This is precisely the mechanism that can make exact Newton hard to globalize and
hard to precondition in strongly shear-thinning regimes.
We therefore use a modified Newton variant that changes only the constitutive
tangent while leaving the nonlinear residual unchanged.

\subsection{Residual functional and exact G\^ateaux derivative}
\label{sec:newton:gateaux}

Let $\mathcal R_{\tau h}(\bm U;\bm W)$ denote the discrete residual
\eqref{eq:disc-residual}.
Then $\boldsymbol u_{\tau h}\in \boldsymbol X_{\tau h}^k$ solves \eqref{eq:disc} if and
only if
\begin{equation}\label{eq:newton-res}
  \mathcal R_{\tau h}(\boldsymbol u_{\tau h};\bm W)=0
  \qquad \forall\,\bm W\in \boldsymbol X_{\tau h}^k.
\end{equation}
The Newton-type correction equation takes the form
$\mathcal J_m(\Delta \bm U^m,\bm W)=\mathcal R_{\tau h}(\bm U^m;\bm W)$ with an
additive update.
Let $\mathcal A_{\tau h}$ denote the left-hand side form of \eqref{eq:disc}, so
that $\mathcal R_{\tau h}$ is the right-hand side minus this left-hand side.
For $\bm U,\Delta\bm U\in \boldsymbol X_{\tau h}^k$, and for states
$\bm U$ at which the directional derivative exists, define
\begin{equation}\label{eq:gateaux-def}
  \mathcal A'_{\tau h}(\bm U)(\Delta\bm U,\bm W)
  \coloneq
  \left.\frac{\mathrm d}{\mathrm d\eps}
  \mathcal A_{\tau h}(\bm U+\eps\Delta\bm U;\bm W)\right|_{\eps=0}
  \qquad \forall\,\bm W\in \boldsymbol X_{\tau h}^k.
\end{equation}
Equivalently,
\[
  D_{\bm U}\mathcal R_{\tau h}(\bm U;\bm W)(\Delta\bm U)
  =
  -\mathcal A'_{\tau h}(\bm U)(\Delta\bm U,\bm W).
\]

For the smooth contributions to the discrete residual associated with
Problem~\ref{Prob:DVP-pns}, the Newton derivative is evaluated exactly.
It consists of the broken DG-in-time term, the differentiated viscous block,
the differentiated convective block, and the mixed pressure-velocity coupling,
including the Nitsche boundary contributions on~$\GhD$.
Here $\mathcal A'_{\nu,\tau h}(\uu)$ denotes the derivative of the viscous
volume term together with the state-dependent viscous boundary-flux term, while
$\mathcal A'_{c,\tau h}(\uu)$ denotes the derivative of the convective
volume and boundary terms that are differentiated in the Newton step.
The factors $\yneg{\uu\cdot\nn}$ on~$\GhD$ and
$\abs{\hat{\uu}\cdot\nn_F}$ in the CIP stabilization are only
Lipschitz with respect to the advecting field and are therefore lagged in
Variant~exN and Variant~modN; see \Cref{sec:newton:exact}.
With this convention, the residual equation itself is unchanged, and the only
genuinely nonlinear elliptic contribution is the constitutive linearization
$\mathcal A'_{\nu,\tau h}(\uu)$.

Given an iterate $\bm U^m\in \boldsymbol X_{\tau h}^k$, we compute a correction
$\Delta\bm U^m\in \boldsymbol X_{\tau h}^k$ from
\begin{equation}\label{eq:nonlinear-step}
  \mathcal J_m(\Delta\bm U^m,\bm W)
  = \mathcal R_{\tau h}(\bm U^m;\bm W)
  \qquad \forall\,\bm W\in \boldsymbol X_{\tau h}^k,
\end{equation}
and update
\[
  \bm U^{m+1}=\bm U^m+\lambda_m\Delta\bm U^m,
  \qquad \lambda_m\in(0,1].
\]
The methods differ only in the choice of the form $\mathcal J_m$. At a fixed
iterate $\uu^m$ set
\[
  \bm A^m \coloneq \DD\uu^m,
  \qquad
  \Delta^{m,2}\coloneq \reg^2+\abs{\bm A^m}^2,
  \qquad
  \mu^m \coloneq \nu(\Delta^{m,2})^{\frac{\pn-2}{2}},
  \qquad
  \eta^m \coloneq \nu_\infty+\mu^m.
\]
At a glance, Variant~Pic lags all state-dependent coefficients and omits the
constitutive derivative, Variant~exN differentiates all smooth residual terms
and lags only the nonsmooth inflow and CIP weights, and Variant~modN agrees
with Variant~exN except that the exact constitutive tangent is replaced by a
stress-clipped symmetric rank-one surrogate.

\subsection{Picard iteration (Variant Pic)}
\label{sec:newton:picard}

Picard freezes all state-dependent coefficients at the current iterate.
For the constitutive term this amounts to the tangent
\begin{equation}\label{eq:local-tangent-picard}
  \mathcal T_{\Pic}^m(\bm B)\coloneq \eta^m \bm B.
\end{equation}
That is, the effective viscosity is evaluated at the current iterate and the
derivative contribution of the stress law is omitted.
In addition, the advecting velocity, the inflow factor
$\yneg{\uu\cdot\nn}$ on $\GhD$, and the state-dependent CIP weight are
lagged at $\bm U^m$.
Hence Picard is a fixed-point iteration of the form \eqref{eq:nonlinear-step}
with a form $\mathcal J_m$.

\subsection{Exact Newton (Variant exN)}
\label{sec:newton:exact}

Recall \eqref{eq:constitutive-law}.
For $\reg>0$, $\bm B\in\mathbb{R}^{d\times d}_{\sym}$, the mapping $\bm A\mapsto \Sstress(\bm A)$ is continuously
differentiable, with Fr\'echet derivative
\begin{equation}\label{eq:pns-dstress}
  \DStress{\bm A}{\bm B}
  = \etaeff{\bm A}\,\bm B
  + \nu(\pn-2)\bigl(\reg^2+\abs{\bm A}^2\bigr)^{\frac{\pn-4}{2}}
    (\bm A:\bm B)\,\bm A.
\end{equation}
The exact constitutive tangent at $\bm A^m$ is the linear map $\mathcal
T_{\exact}^m(\bm B)\colon \mathbb{R}^{d\times d}_{\sym}\to \mathbb{R}^{d\times
  d}_{\sym}$ defined by
\begin{equation}\label{eq:local-tangent-exact}
\begin{split}
  \mathcal T_{\exact}^m(\bm B)
  &= \eta^m\bm B
   + \nu(\pn-2)(\Delta^{m,2})^{\frac{\pn-4}{2}}(\bm A^m:\bm B)\,\bm A^m \\
  &= \underbrace{\eta^m\bm B}_{\text{isotropic}}
    + \underbrace{(\pn-2)\frac{\mu^m}{\Delta^{m,2}}(\bm A^m:\bm B)\,\bm A^m}_{\text{rank-one
    correction in $\bm A^m$-direction}}.
\end{split}
\end{equation}
By Appendix~\ref{app:tangent-spectrum}, the eigenvalues of $\mathcal
T_{\exact}^m$ are
$\lambda_\perp=\eta^m$ on the subspace defined by $\{\bm B:\bm A^m:\bm B=0\}$ and
\begin{equation}\label{eq:exact-par-eig}
  \lambda_\parallel
  = \eta^m + (\pn-2)\frac{\mu^m}{\Delta^{m,2}}\abs{\bm A^m}^2
  = \nu_\infty+\mu^m
    \frac{\reg^2+(\pn-1)\abs{\bm A^m}^2}{\Delta^{m,2}}
\end{equation}
in the $\bm A^m$-direction.
From Appendix~\ref{app:tangent-spectrum}, in particular
\Cref{lem:rankone-tangent}, we obtain, for $1<\pn < 2$,
\begin{equation}\label{eq:ell-exact}
  \bigl(\nu_\infty+(\pn-1)\mu^m\bigr)\abs{\bm B}^2
  \le (\mathcal T_{\exact}^m(\bm B):\bm B)
  \le (\nu_\infty+\mu^m)\abs{\bm B}^2.
\end{equation}
For $\abs{\bm A^m}\gg\reg$, the anisotropy ratio behaves like
\begin{equation}\label{eq:local-cond-exact}
  \kappa(\mathcal T_{\exact}^m)
  \approx
  \frac{\nu_\infty+\mu^m}{\nu_\infty+(\pn-1)\mu^m},
\end{equation}
which deteriorates as $\pn\downarrow 1$ and $\mu^m\gg\nu_\infty$ due to the
factor $p-1$ in the smallest-eigenvalue direction. This is the source of
ill-conditioning in the shear-thinning regime.

For all smooth residual contributions, Variant~exN uses the G\^ateaux derivative,
\[
  \mathcal J_m(\Delta\bm U,\bm W)
  = \mathcal A'_{\tau h}(\bm U^m)(\Delta\bm U,\bm W),
\]
with constitutive part induced by \eqref{eq:local-tangent-exact}.
Thus Variant~exN is exact for the differentiable part of the residual.
The factors $\yneg{\uu\cdot\nn}$ on~$\GhD$ and
$\abs{\hat{\uu}\cdot\nn_F}$ in \eqref{eq:scip-def} are only Lipschitz in
the advecting field. Accordingly, in the implementation these weights are
lagged at the current iterate rather than differentiated.

\subsection{Modified Newton (Variant modN)}
\label{sec:newton:modified}
To obtain a more robust constitutive surrogate, we replace
$\mathcal T_{\exact}^m$ by a stress-clipped symmetric rank-one modification.
Define
\[
\bm \sigma^m \coloneq \mu^m \bm A^m,
\qquad
\hat{\bm\sigma}^m \coloneq s^m\bm\sigma^m,
\]
with
\[
s^m \coloneq
\begin{cases}
\min\{1,\sigma_{\max}/\abs{\bm\sigma^m}\}, & \bm\sigma^m\neq \bm 0,\\
0, & \bm\sigma^m=\bm 0,
\end{cases}
\]
so that $\abs{\hat{\bm\sigma}^m}\le \sigma_{\max}$.
Further set
\(
\theta^m \coloneq \frac{(\pn-2)}{\mu^m\Delta^{m,2}} .
\)
Since $\hat{\bm\sigma}^m=s^m\bm\sigma^m$, the modified constitutive tangent is
\begin{equation}\label{eq:local-tangent-clip}
\mathcal T_{\clip}^m(\bm B)
= \eta^m\bm B + s^m\theta^m(\bm\sigma^m:\bm B)\,\bm\sigma^m.
\end{equation}
Equivalently, using $\bm\sigma^m=\mu^m\bm A^m$,
\[
\mathcal T_{\clip}^m(\bm B)
= \eta^m\bm B + s^m(\pn-2)\frac{\mu^m}{\Delta^{m,2}}(\bm A^m:\bm B)\,\bm A^m.
\]
Hence Lemma~\ref{lem:rankone-tangent} applies and,
if $\bm\sigma^m\neq0$, the eigenvalues are
\[
  \lambda_\perp=\eta^m,
  \qquad
  \lambda_\parallel
  = \eta^m + s^m\theta^m |\bm\sigma^m|^2,
\]
with $\lambda_\perp$ on the subspace
$\{\bm B:\bm\sigma^m:\bm B=0\}$ and
$\lambda_\parallel$ in the direction $\bm\sigma^m$.
For $\abs{\bm A^m}\gg\reg$, the anisotropy ratio behaves like
\begin{equation}\label{eq:local-cond-clip}
  \kappa(\mathcal T_{\clip}^m)
  \approx
  \frac{\nu_\infty+\mu^m}
       {\nu_\infty+\bigl(1-s^m(2-\pn)\bigr)\mu^m},
\end{equation}
so that $s^m=1$ recovers \eqref{eq:local-cond-exact}, whereas $s^m=0$ yields
the isotropic Picard value $\kappa(\mathcal T_{\clip}^m)=1$. Hence, clipping
replaces the factor $\pn-1$ by $1-s^m(2-\pn)$ in the
smallest-eigenvalue direction and thereby weakens the anisotropy.

Modified Newton keeps the residual \eqref{eq:newton-res} unchanged and replaces
only the exact constitutive tangent.
More precisely, $\mathcal T_{\exact}^m$ is replaced by
\eqref{eq:local-tangent-clip} in the viscous volume term and in the matching
state-dependent boundary flux term.
All remaining contributions, in particular the time derivative, the mixed
pressure-velocity coupling, the convective linearization, and the
fixed-coefficient Nitsche terms, are treated as in Variant~exN.
The method is therefore a quasi-Newton iteration.

\subsection{Coercivity of the viscous-Nitsche term}
The next corollary records the coercivity property used later in the algebraic
discussion: under pointwise lower and upper bounds on the constitutive tangent,
the linearized viscous-Nitsche term remains coercive.
Let $\eta_D\coloneq \etaeff{\DD\hat{\boldsymbol g}_D}$ on $\GhD$ (hence
$\eta_D\ge \nu_\infty$), and define for $\uu_h,\vv_h\in\Vh$
\begin{equation}\label{eq:am-def}
  \begin{aligned}
    a_{\star}^m(\uu_h,\vv_h)
    \coloneqq\;&
                 \Ip{\mathcal T_\star^m(\DD\uu_h)}{\DD\vv_h}
                 -\Ip{\mathcal T_\star^m(\DD\uu_h)\nn}{\vv_h}_{\GhD}
                 -\Ip{\uu_h}{\DStress{\DD\hat{\boldsymbol g}_D}{\DD\vv_h}\nn}_{\GhD} \\
               &\quad
                 +\gOne\,\Ip{\hGD^{-1}\eta_D\,\uu_h}{\vv_h}_{\GhD}
                 +\gTwo\,\Ip{\hGD^{-1}(\uu_h\cdot\nn)}{(\vv_h\cdot\nn)}_{\GhD}.
  \end{aligned}
\end{equation}
Here $\hat{\boldsymbol g}_D$ is the lifting of $\boldsymbol g_D$ to $\Omega$.

\begin{corollary}[Coercivity of the linearized viscous-Nitsche term]\label{cor:linearized-viscous-coercive}
Fix a Newton state $\bm U^m$ and a tangent
$\mathcal T_\star^m:\mathbb{R}^{d\times d}_{\sym}\to\mathbb{R}^{d\times d}_{\sym}$.
Assume there exists $\overline\eta\ge \nu_\infty>0$ such that, for a.e.\ $x\in\Omega$ and all
$\bm B\in\mathbb{R}^{d\times d}_{\sym}$,
\begin{equation}\label{eq:tangent-bounds}
  \nu_\infty |\bm B|^2 \le \mathcal T_\star^m(\bm B):\bm B,
  \qquad
  |\mathcal T_\star^m(\bm B)|+\bigl|\DStress{\DD\hat{\boldsymbol g}_D}{\bm B}\bigr|
  \le \overline\eta\,|\bm B|.
\end{equation}
For the tangents used in this paper, these bounds hold in the uniformly elliptic
regime $\nu_\infty>0$.
There exists $C_{\mathrm{tr}}>0$, depending only on shape regularity and on
$\overline\eta/\nu_\infty$, such that for all $\uu_h\in\Vh$,
\begin{equation}\label{eq:am-coercive}
  a_{\star}^m(\uu_h,\uu_h)
  \ge \frac{\nu_{\infty}}{2}\,\|\DD\uu_h\|_{L^2(\Omega)}^2
  +\bigl(\gOne\,\nu_{\infty}-C_{\mathrm{tr}}\bigr)\,\|\hGD^{-\frac12}\uu_h\|_{L^2(\GhD)}^2
  +\gTwo\,\|\hGD^{-\frac12}(\uu_h\cdot\nn)\|_{L^2(\GhD)}^2.
\end{equation}
In particular, if $\gOne\ge 2C_{\mathrm{tr}}/\nu_{\infty}$ then $a_\star^m$ is
coercive on $\Vh$.
Moreover, inhomogeneous Dirichlet data enter only through the fixed
coefficients $\eta_D$ and $\DStress{\DD\hat{\boldsymbol g}_D}{\cdot}$; once the
bounds in \eqref{eq:tangent-bounds} hold, the coercivity argument is unchanged.
\end{corollary}

\begin{proof}
Set $\vv_h=\uu_h$ in \eqref{eq:am-def}.
By \eqref{eq:tangent-bounds},
$\Ip{\mathcal T_\star^m(\DD\uu_h)}{\DD\uu_h}\ge
\nu_\infty\|\DD\uu_h\|_{L^2(\Omega)}^2$.
Further, by the upper bound in \eqref{eq:tangent-bounds}, Cauchy-Schwarz, and
a discrete trace inequality,
\begin{equation*}
  \begin{split}
    &\bigl|\Ip{\mathcal T_\star^m(\DD\uu_h)\nn}{\uu_h}_{\GhD}\bigr|
    +\bigl|\Ip{\uu_h}{\DStress{\DD\hat{\boldsymbol g}_D}{\DD\uu_h}\nn}_{\GhD}\bigr|\\
    &\qquad\le
    C_{\mathrm{tr}}^{1/2}\|\DD\uu_h\|_{L^2(\Omega)}
    \|\hGD^{-1/2}\uu_h\|_{L^2(\GhD)}.
  \end{split}
\end{equation*}
Young's inequality yields
\begin{equation*}
\begin{split}
\bigl|\Ip{\mathcal T_\star^m(\DD\uu_h)\nn}{\uu_h}_{\GhD}\bigr|
&+\bigl|\Ip{\uu_h}{\DStress{\DD\hat{\boldsymbol g}_D}{\DD\uu_h}\nn}_{\GhD}\bigr|\\
&\quad\le \frac{\nu_\infty}{2}\|\DD\uu_h\|_{L^2(\Omega)}^2
   +C_{\mathrm{tr}}\|\hGD^{-1/2}\uu_h\|_{L^2(\GhD)}^2.
\end{split}
\end{equation*}
Finally, $\eta_D\ge \nu_\infty$ implies
$\gOne\,\Ip{\hGD^{-1}\eta_D\,\uu_h}{\uu_h}_{\GhD}\ge
 \gOne\nu_\infty\|\hGD^{-1/2}\uu_h\|_{L^2(\GhD)}^2$,
and the $\gTwo$-term is nonnegative.
Collecting the estimates gives \eqref{eq:am-coercive}.
\end{proof}

\section{Algebraic structure and assembly}
\label{sec:algebraic}
We now make the algebraic structure of one time step explicit.
This is the basis for matrix-free Jacobian actions and for the monolithic
space-time multigrid preconditioner used within the time-marching realization.

\subsection{Algebraic system}
Fix a time step $\In$~\eqref{eq:In-Sn}.
Let $\{\boldsymbol\varphi_i\}_{i=1}^{M_v}$ and $\{\psi_j\}_{j=1}^{M_p}$ be
bases of $\Vh$ and $Q_h$.
On $\In$ we use the Lagrange basis
$\{\ell_\mu\}_{\mu=1}^{k+1}\subset\Pk(\In)$ associated with the right-sided
Gauss-Radau nodes $\{t_n^\mu\}_{\mu=1}^{k+1}$.
Expanding
\[
  \uu_{\tau h}|_{\Sn} = \sum_{\mu=1}^{k+1} \sum_{i=1}^{M_v} V_{\mu i}^n\,\ell_\mu(t)\,\boldsymbol\varphi_i(x),
  \qquad
  \pi_{\tau h}|_{\Sn} = \sum_{\mu=1}^{k+1} \sum_{j=1}^{M_p} P_{\mu j}^n\,\ell_\mu(t)\,\psi_j(x),
\]
yields a coupled nonlinear system on $\Sn=\Omega\times I_n$,
\begin{equation}\label{eq:slab-system}
  \bm R_n(\bm U^n) = \bm 0,
  \qquad
  \bm U^n \coloneq \bigl((\bm V_\mu^n,\bm P_\mu^n)\bigr)_{\mu=1}^{k+1}\in\mathbb R^{(k+1)(M_v+M_p)},
\end{equation}
where $\bm V_\mu^n\in\mathbb R^{M_v}$ and $\bm P_\mu^n\in\mathbb R^{M_p}$
collect the spatial coefficients at the temporal node $t_n^\mu$.
Let $\mathcal R_n$ denote the time-step residual functional obtained from
\eqref{eq:disc-residual} by restricting to $\Sn$ and treating the left trace
$\uu_{\tau h}^-(t_{n-1})$ as given data.
The associated residual vector
$\bm R_n:\mathbb R^{(k+1)(M_v+M_p)}\to\mathbb R^{(k+1)(M_v+M_p)}$ is defined by
\[
  (\bm R_n(\bm U^n))_{(\mu,i)} \coloneq
  \mathcal R_n \left(\bm U^n;\,\ell_\mu\,\boldsymbol\varphi_i,0\right),
  \qquad
  (\bm R_n(\bm U^n))_{(\mu,j)} \coloneq
  \mathcal R_n \left(\bm U^n;\,0,\ell_\mu\,\psi_j\right).
\]
Then, finding an approximation to \Cref{Prob:DVP-pns} is equivalent to solving
\eqref{eq:slab-system} for all $\Sn$, $n=1,\dots,N$. The global algebraic system
corresponding to \Cref{Prob:DVP-pns} is obtained by stacking the time-step
vectors $\bm U^n$ and adding the DG jump couplings between consecutive time
steps. In the time-marching realization, the left trace
$\uu_{\tau h}^-(t_{n-1})$ enters $\In$ only through the DG jump term and is
treated as known when solving on $\In$. Accordingly, the implementation advances
from $t_{n-1}$ to $t_n$ by solving the monolithic local
system~\eqref{eq:slab-system} on each time step.

\subsection{Tensor-product structure and matrix-free operator evaluation}
On each $\Sn$ we evaluate time integrals involving state-dependent coefficients
by the right-sided $(k{+}1)$-point Gauss-Radau rule underlying the temporal
Lagrange basis.
This preserves the endpoint $t_n$ needed by the DG jump term and is exact on
$\mathbb P_{2k}(\In)$.
For nonlinear constitutive laws this introduces a quadrature error; see
Appendix~\ref{sec:gr-underint}.

Let $\hat I=(0,1)$ and let $\{\hat\xi_i\}_{i=1}^{k+1}$ be the Gauss-Radau
Lagrange basis on $\hat I$.
The temporal matrices on $\In$ are obtained from the reference interval by
affine pullback; for brevity we denote them again by
$\bm M_t\in\mathbb R^{k+1\times k+1}$,
$\bm K_t\in\mathbb R^{k+1\times k+1}$, and
$\bm m_t\in\mathbb R^{k+1}$.
On the reference interval they are given by
\begin{equation*}
  {(\bm M_t)}_{i,\,j} \coloneq
  \int_0^1 \hat\xi_j(\hat t)\hat\xi_i(\hat t)\drv\hat t,\quad
  {(\bm K_t)}_{i,\,j} \coloneq
  \int_0^1 \hat\xi_j'(\hat t)\hat\xi_i(\hat t)\drv\hat t
  + \hat\xi_j(0)\hat\xi_i(0),\quad
  {(\bm m_t)}_{i} \coloneq  \hat\xi_i(0),
\end{equation*}
for $i,j=1,\dots,k+1$.
The endpoint term $\hat\xi_j(0)\hat\xi_i(0)$ encodes the left-sided DG jump.
Because $\hat\xi_i\hat\xi_j\in \mathbb P_{2k}(\hat I)$ and the $(k{+}1)$-point
Gauss-Radau rule is exact on $\mathbb P_{2k}(\hat I)$, the mass matrix
$\bm M_t$ coincides with the diagonal matrix of Gauss-Radau weights in this
nodal basis. Note that the left endpoint is not a quadrature node of the
right-sided Gauss-Radau rule.

Let $\bm M_x\in\mathbb R^{M_v\times M_v}$ be the spatial velocity mass matrix,
\[
  (\bm M_x)_{i,j} \coloneq  (\boldsymbol\varphi_j,\boldsymbol\varphi_i)_{L^2(\Omega)}.
\]
Write the time-step unknown in velocity-pressure block form
$\bm U^n=(\bm V^n,\bm P^n)$ with
\[
  \bm V^n=(\bm V_1^n,\ldots,\bm V_{k+1}^n)\in\mathbb R^{(k+1)M_v},\qquad
  \bm P^n=(\bm P_1^n,\ldots,\bm P_{k+1}^n)\in\mathbb R^{(k+1)M_p}.
\]
For each temporal node $\mu$, let
$\bm F_\mu(\bm V_\mu^n,\bm P_\mu^n)\in\mathbb R^{M_v+M_p}$ denote the assembled
\emph{spatial residual vector} at $t=t_n^\mu$, with the same sign convention as
$\bm R_n$, i.e.\ right-hand side minus spatial left-hand side, excluding the DG
time derivative and the DG jump.
Let $\bm J_\mu$ denote the corresponding spatial linearization block used in the
correction equation \eqref{eq:nk-newton-lin}, i.e.\ the matrix representation of
the selected left-hand-side linearization on the node $t_n^\mu$.
Equivalently, $\bm J_\mu=-D\bm F_\mu$ for the chosen nonlinear variant.
For Variant~exN, $\bm J_\mu$ differentiates all smooth left-hand-side terms,
with the nonsmooth inflow and CIP weights lagged as described in
\Cref{sec:newton:exact}. For Variant~Pic and Variant~modN, $\bm J_\mu$ is the
analogous spatial linearization obtained from the same residual by the
corresponding constitutive replacement from \Cref{sec:newton}.

Because each $\bm F_\mu$ is assembled at the single node $t_n^\mu$, it depends
only on $(\bm V_\mu^n,\bm P_\mu^n)$ and not on nodal values at any other
Gauss-Radau node.
Consequently, the spatial Jacobian blocks decouple across temporal nodes and
form the block diagonal matrix
$\diag(\bm J_1,\ldots,\bm J_{k+1})$.
Define the stacked spatial residual
$\bm F(\bm U^n)\coloneq (\bm F_1,\ldots,\bm F_{k+1})$ and let
$\bm v_{n-1}^-\in\mathbb R^{M_v}$ denote the coefficient vector of the left
trace $\uu_{\tau h}^-(t_{n-1})$.

Then the time-step residual and Jacobian can be written as
\begin{equation}\label{eq:slab-underint}
  \bm R_n(\bm U^n)
  =
  -
  \begin{bmatrix}
    (\bm K_t\otimes \bm M_x)\,\bm V^n - (\bm m_t\otimes \bm M_x)\,\bm v_{n-1}^-\\[0.2em]
    \bm 0
  \end{bmatrix}
  +
  (\bm M_t\otimes \bm I_{M_v+M_p})\,\bm F(\bm U^n),
\end{equation}
\begin{equation}\label{eq:slab-underint-jac}
  \bm J_n(\bm U^n)
  =
  \begin{bmatrix}
    (\bm K_t\otimes \bm M_x) &\bm 0\\[0.2em]
    \bm 0&\bm 0
  \end{bmatrix}
  +
  (\bm M_t\otimes \bm I_{M_v+M_p})\,
  \diag(\bm J_1,\ldots,\bm J_{k+1}).
\end{equation}
Hence a matrix-free Jacobian action reduces to independent spatial kernel
applications $\bm J_\mu(\cdot)$ at the Gauss-Radau nodes, followed by temporal
scaling with $\bm M_t$ and the dense temporal coupling from $\bm K_t$.

Each spatial Jacobian block
$\bm J_\mu\in\mathbb R^{(M_v+M_p)\times(M_v+M_p)}$ has the mixed form
\begin{equation}\label{eq:Jmu-block}
  \bm J_\mu
  =
  \begin{bmatrix}
    \bm A_\mu & \bm B_\mu^\top\\
    \bm B_\mu & \bm 0
  \end{bmatrix},
\end{equation}
where $\bm B_\mu$ denotes the discrete incompressibility block, including the
Nitsche boundary contribution on $\GhD$, and $\bm B_\mu^\top$ the
corresponding pressure coupling block.
Moreover,
\begin{equation}\label{eq:Amu-decomp}
  \bm A_\mu = \bm A_{\nu,\mu}(\bm U^n) + \bm A_{c,\mu}(\bm U^n) + \bm A_{\mathrm{CIP},\mu}(\bm U^n).
\end{equation}
Here $\bm A_{\nu,\mu}$ collects the linearized viscous volume term and the
state-dependent boundary flux term with the tangent variant selected in
\Cref{sec:newton}; the fixed-coefficient Nitsche terms are unchanged from the
discretization.
Further, $\bm A_{c,\mu}$ denotes the linearization of convection, including the
inflow boundary term, and $\bm A_{\mathrm{CIP},\mu}$ the linearization of the
CIP stabilization.
If the advecting field in the CIP weight is lagged, then
$\bm A_{\mathrm{CIP},\mu}$ is symmetric positive semidefinite, because
\eqref{eq:scip-def} is an $L^2(F)$ inner product of identical jumps with
nonnegative weight.
Thus lagged CIP cannot reduce the coercivity of the velocity terms.

A representative-state surrogate is used only inside the local Vanka patch
solves of the multigrid preconditioner.
The nonlinear residual and the global matrix-free Jacobian action always use the
physical operator associated with the selected nonlinear variant.
The surrogate affects only local patch assembly inside the smoother;
see \Cref{sec:mg-framework}.

\section{Scalable algebraic nonlinear and linear solvers}
\label{sec:newton-krylov}

The nonlinear iterations from \Cref{sec:newton} are realized by time marching in
a matrix-free Newton-Krylov framework, with one monolithic space-time solve per
time step and one space-time multigrid V-cycle as preconditioner.
We restrict the discussion to the aspects specific to the present
$(\pn,\reg)$-Navier-Stokes setting.
On each time step $\In$, the fully implicit discretization yields a nonlinear
system
\[
  \bm R_n(\bm U^n)=\bm 0,
\]
where $\bm U^n$ collects all space-time degrees of freedom on $\Sn$ and
$\bm R_n$ is the algebraic realization of $\mathcal R_{\tau h}(\bm U;\bm W)$ in
the test basis. We write $\bm U^{n,m}$ for the $m$-th nonlinear iterate on the
slab $\Sn$, and $\bm J_n(\bm U^{n,m})\coloneq -D\bm R_n(\bm U^{n,m})$ for the
corresponding algebraic linearization. Thus $\bm J_n$ is the matrix
representation, in the chosen test basis, of the form
$\mathcal J_m$ from \eqref{eq:nonlinear-step}. Given an iterate $\bm U^{n,m}$, we compute a correction
$\Delta\bm U^{n,m}$ from
\begin{equation}\label{eq:nk-newton-lin}
  \bm J_n(\bm U^{n,m})\Delta\bm U^{n,m}
  = \bm R_n(\bm U^{n,m}),
  \qquad
  \bm U^{n,m+1}=\bm U^{n,m}+\lambda_m\Delta\bm U^{n,m}.
\end{equation}
Here $\bm J_n(\bm U^{n,m})$ denotes the time-step linearization induced by the
selected nonlinear variant: the operator for Picard, the Jacobian of the
differentiable residual terms for Variant~exN, and the quasi-Newton operator for
modified Newton (Variant~modN).

\subsection{Krylov solver and preconditioning}

The linear systems \eqref{eq:nk-newton-lin} are solved by right-preconditioned
FGMRES with one STMG V-cycle as preconditioner.
The matrix-free operator evaluation, transfer operators, coarse-grid treatment,
and $hp$ multigrid hierarchy follow
\cite{MargenbergBause2026MonolithicSTMG,MargenbergMunchBause2025hpMGStokes} and
are not repeated here.

We measure algebraic residuals in the block mass-weighted norm
\[
  \norm{(\bm r_v,\bm r_p)}_{\mathcal M}^2
  \coloneq
  \bm r_v^\top \bm M_v \bm r_v+\bm r_p^\top \bm M_p \bm r_p,
\]
where $\bm M_v$ and $\bm M_p$ denote the time-step velocity and pressure mass blocks.
For exact Newton and modified Newton, FGMRES is terminated by the inexact-solve
condition
\begin{equation}\label{eq:nk-inexact-newton}
  \norm{\bm J_n(\bm U^{n,m})\Delta\bm U^{n,m}-\bm R_n(\bm U^{n,m})}_{\mathcal M}
  \le
  \eta_m \norm{\bm R_n(\bm U^{n,m})}_{\mathcal M},
\end{equation}
with an adaptive forcing term $\eta_m$ (cf.~\cite{MargenbergBause2026MonolithicSTMG}).
For Picard, we use the same Krylov solver with a prescribed relative tolerance.
Concrete parameter values are reported in \Cref{sec:numerics}.

\subsection{Globalization and stopping criteria}

For exact Newton and modified Newton, we use Armijo backtracking on the
current slab $\Sn$, with
\[
  \phi_n(\bm U)\coloneq \tfrac12 \norm{\bm R_n(\bm U)}_{\mathcal M}^2.
\]
Picard steps are taken with $\lambda_m=1$ unless additional residual-based
damping is required.

The nonlinear iteration is terminated once
\begin{equation}\label{eq:nk-stop}
  \norm{\bm R_n(\bm U^{n,m})}_{\mathcal M}\le \tau_{\mathrm{abs}}
  \qquad\text{or}\qquad
  \norm{\bm R_n(\bm U^{n,m})}_{\mathcal M}
  \le
  \tau_{\mathrm{rel}}\norm{\bm R_n(\bm U^{n,0})}_{\mathcal M}.
\end{equation}
We additionally impose a maximal number of nonlinear iterations and declare the
time-step solve unsuccessful if the line search or the linear solver fails to meet
the prescribed tolerances.

\subsection{Multigrid preconditioning and surrogate patch assembly}
\label{sec:mg-framework}

We use one space-time multigrid V-cycle as preconditioner in each Krylov step.
Here ``space-time multigrid'' refers to multigrid applied to the monolithic
system on one time step $\Sn$, including all temporal basis functions on that
interval; we do not apply multigrid to the whole time horizon at once.
The general $hp$ space-time multigrid construction, including the level
hierarchy, transfer operators, and coarse-grid treatment, is the same as in
\cite{MargenbergBause2026MonolithicSTMG,MargenbergMunchBause2025hpMGStokes}.
The only approximation specific to the present nonlinear setting is a
finest-level surrogate assembly of local patch matrices. On coarser levels the
patch matrices are taken from the exact Petrov-Galerkin operators.

\paragraph{Hierarchy and coarse operators}
A level $\ell\in\{0,\dots,L\}$ is specified by a spatial mesh $\mathcal M_\ell$
and a temporal DG degree $k_\ell$, with finest-level operator
\[
  \bm A_L \coloneq \bm J_n(\bm U^{n,m}).
\]
The interlevel transfer operators are the tensor-product prolongation and
restriction operators from
\cite{MargenbergBause2026MonolithicSTMG,MargenbergMunchBause2025hpMGStokes},
and the coarse operators are defined recursively by the Petrov-Galerkin
relation
\[
  \bm A_\ell \coloneq \bm \Pi_\ell \bm A_{\ell+1}\bm P_\ell,
  \qquad \ell=L-1,\dots,0.
\]

\paragraph{Vanka smoothing}
On each level we use one pre- and one post-smoothing step of an overlapping
additive Schwarz method of Vanka type.
For a spatial patch $\omega_K$ with index set $I_K$ of velocity and pressure
degrees of freedom, the associated space-time patch on $\Sn$ contains all
temporal coefficients attached to the unknowns in $I_K$; equivalently, the local
Vanka solve is performed on $\omega_K\times \In$.
With $\bm R_K$ denoting the restriction to this patch, one smoothing step is
\begin{equation}\label{eq:stmg-vanka-update}
  \bm d \leftarrow
  \bm d + \omega \sum_{K\in\mathcal M_\ell}
  \bm R_K^\top
  (\bm R_K \bm A_\ell \bm R_K^\top)^{-1}
  \bm R_K(\bm r-\bm A_\ell \bm d),
\end{equation}
with damping parameter $\omega\in(0,1]$.
The choice of cell patches versus vertex-star patches in space is as in
\cite{MargenbergBause2026MonolithicSTMG,MargenbergMunchBause2025hpMGStokes};
in time we always use the full time-step block.

\paragraph{Exact patch matrices}
The time-step operator $\bm A_\ell$ is never assembled globally.
On the finest level its action is obtained by quadrature-based matrix-free
evaluation; on coarser levels it is defined by Petrov-Galerkin projection.
For every level, the exact patch matrix is
\[
  \bm A_{\ell,K}\coloneq \bm R_K \bm A_\ell \bm R_K^\top.
\]
Only the finest level is locally reassembled from time-node spatial blocks.
There, the state-dependent spatial blocks in \eqref{eq:slab-underint-jac}
depend on the Gauss-Radau node, so exact local patch assembly involves
$k_L+1$ distinct spatial blocks on the same time step.
More precisely, with
\[
  \bm J_{L,\mu}\coloneq -D\bm F_{\mu}(\bm U^{n,m}),
  \qquad \mu=1,\dots,k_L+1,
\]
the finest-level time-local factor contains the block diagonal matrix
\[
  \diag(\bm J_{L,1},\dots,\bm J_{L,k_L+1}).
\]
Consequently, exact finest-level patch assembly would require, for every patch
$K$, the family of spatial patch blocks
\[
  \bm J_{L,\mu,K}\coloneq
  \bm R_K^{\mathrm s}\bm J_{L,\mu}(\bm R_K^{\mathrm s})^\top,
  \qquad \mu=1,\dots,k_L+1,
\]
where $\bm R_K^{\mathrm s}$ denotes the purely spatial patch restriction.

\paragraph{Surrogate patch assembly}
The only approximation in the smoother concerns the node-wise spatial block in
\eqref{eq:slab-underint-jac} on the finest level.
With the exact patch matrix
$\bm A_{L,K}=\bm R_K \bm A_L \bm R_K^\top$ from the preceding paragraph,
all temporal coefficient matrices from
\eqref{eq:slab-underint}--\eqref{eq:slab-underint-jac}, in particular
$\bm M_t$ and $\bm K_t$, are kept unchanged.

Choose a representative time $t_n^{\mathrm{rep}}\in \In$ and define a
representative spatial linearization block by evaluating all state-dependent
coefficients entering the locally reassembled finest-level patch operator once
at $t_n^{\mathrm{rep}}$ and then freezing them over the time step.
In the present computations we use the midpoint of $\In$.

At the level of locally reassembled finest-level patch matrices, the surrogate
replaces the family of time-node blocks by one representative block,
\begin{equation}\label{eq:surrogate-diag}
  \diag(\bm J_{L,1},\dots,\bm J_{L,k_L+1})
  \quad\leadsto\quad
  \bm I_{k_L+1}\otimes \bm J_L^{\mathrm{rep}},
\end{equation}
with all remaining temporal factors unchanged.
The corresponding surrogate finest-level patch matrix is denoted by
\[
  \tilde{\bm A}_{L,K}.
\]

Thus the surrogate replaces the $k_L+1$ time-dependent spatial blocks on one
time step by one representative spatial block, while preserving the temporal
coupling encoded by the original time-step matrices.
The approximation is confined to the local finest-level patch solve:
\[
  (\bm R_K \bm A_L \bm R_K^\top)^{-1}
  \quad\leadsto\quad
  \tilde{\bm A}_{L,K}^{-1}.
\]
The global defect $\bm r-\bm A_\ell \bm d$, the matrix-free operator action,
and the coarse-grid operators remain exact. On coarser levels we use the exact
patch matrices $\bm A_{\ell,K}$ extracted from the Petrov-Galerkin operators.
We show in Appendix~\ref{app:patch-surrogate} that the finest-level patch
perturbation is controlled by the time-step size.

\paragraph{Interpretation and refresh strategy}
The surrogate is a purely local coefficient freezing in the finest-level Vanka
patch matrices. It does not change the nonlinear residual or the matrix-free
global Jacobian action; it changes only the local patch solves in the multigrid
smoother and thus reduces smoother setup costs.

Within one FGMRES solve, the patch matrices used by the Vanka smoother and
their factorizations are kept fixed. They are assembled at the beginning of
each Newton or modified Newton step. Before a subsequent linear solve, the
finest-level surrogate patch matrices are refreshed only if the preceding solve
has deteriorated. Following \cite{MargenbergBause2026MonolithicSTMG}, we
trigger a rebuild when
\[
  \frac{\norm{\bm R_n(\bm U^{n,m+1})}_{\mathcal M}}
  {\norm{\bm R_n(\bm U^{n,m})}_{\mathcal M}}>\rho,
\]
or when the FGMRES iteration count from the preceding linear solve exceeds
twice that of the previous nonlinear step; see \Cref{sec:numerics} for the
parameters.

\section{Numerical experiments}
\label{sec:numerics}
We first verify the fully discrete implementation by a manufactured-solution
convergence test in \Cref{sec:conv-ns} and then study a genuinely time-dependent
benchmark problem in \Cref{sec:dfg-bench}. Throughout, the emphasis is on the fully implicit
space-time solver: nonlinear robustness, linear iteration counts, time dependence across time steps, and computational efficiency.

In space we use inf-sup stable pairs
$\mathbb{Q}_{2}/\mathbb{P}_1^{\mathrm{disc}}$, and in time a DG$(1)$ discretization.
Nonlinear systems are solved by the framework of
\Cref{sec:newton,sec:newton-krylov}. The manufactured-solution and
parameter studies in \Cref{sec:conv-ns} compare Picard, Newton, and modified Newton solvers
in the same algebraic framework. For the time-dependent DFG benchmark we focus
on the modified Newton, which was the most reliable option in the preceding tests.
We terminate the nonlinear iteration when the residual satisfies an absolute tolerance of $10^{-12}$ and a relative tolerance of
$10^{-10}$. These strict criteria are chosen to verify that the modified tangents do not introduce an accuracy barrier, i.e.,
they permit reduction of the residual to near machine precision rather than stagnating at a tangent-induced floor.
We use Eisenstat-Walker forcing and Armijo
backtracking as described in \Cref{sec:newton-krylov}. Each linearized subproblem is treated by FGMRES preconditioned with a single
$V$-cycle geometric multigrid with Vanka-type smoothing.

The numerical experiments were performed on a workstation with an Intel Xeon Gold 6254 CPU and 396GB of RAM.
The implementation is based on the deal.II library and openly available
at~\cite{margenberg_scalable_2026}. We call the method \textbf{$h$-robust in Newton} if the nonlinear
iteration counts stay bounded under mesh refinement, and \textbf{$h$-robust in
  FGMRES} if the linear iteration counts remain bounded under the same
refinement. We use two pre- and two post-smoothing steps throughout this
section. To quantify the effect of inertia we vary $\nu_\infty$. We choose a
characteristic length $L$ and velocity scale $U$ (e.\,g.\ the channel width and
mean inflow velocity for the ``flow around a cylinder'' benchmark) and evaluate
the apparent viscosity at the characteristic shear rate
$\dot\gamma_c\coloneq U/L$. For the constitutive
law~\eqref{eq:constitutive-law}, the associated scalar apparent viscosity
(cf.~\cite[Section~4.1]{BirdArmstrongHassager1987DPL1}) and the apparent
Reynolds number are
\begin{equation}\label{eq:reapp}
  \eta_{\mathrm{app}}(\dot\gamma)
  \coloneq \nu_\infty + \nu\bigl(\reg^2+\dot\gamma^2\bigr)^{\frac{\pn-2}{2}},\quad\mathrm{Re}_{\mathrm{app}}
  \coloneq \frac{U L}{\eta_{\mathrm{app}}(\dot\gamma_c)}\,.
\end{equation}

\subsection{\label{sec:conv-ns}Convergence test}
\emph{Problem setup.} We consider a model problem on the space-time domain $\Omega\times \It = {[0,1]}^2\times [0,1]$ with prescribed velocity
$\uu \colon \Omega\times \It \to \mathbb{R}^2$ and pressure
$\pi \colon \Omega\times \It \to \mathbb{R}$ given by
\begin{subequations}\label{eq:conv-test}
  \begin{align}
    \label{eq:conv-test-v}
    \uu(\mathbf{x},\,t) &= \sin(t) \begin{pmatrix} \sin^2(\pi x) \sin(\pi y) \cos(\pi y) \\
      -\sin(\pi x) \cos(\pi x) \sin^2(\pi y) \end{pmatrix},\\
    \label{eq:conv-test-p} \pi(\mathbf{x},\,t) &= \sin(t) \sin(\pi x) \cos(\pi x)
                                               \sin(\pi y) \cos(\pi y)\,.
  \end{align}
\end{subequations}
We choose the forcing term $\ff$ such that $(\uu,\pi)$ satisfies the
$(\pn,\reg)$-Navier-Stokes system~\eqref{eq:pns-strong} with stress
law~\eqref{eq:constitutive-law}. The velocity field in~\eqref{eq:conv-test-v} is
divergence-free. The purpose of this test is verification of the fully discrete
implementation; the example is nevertheless nontrivial because the velocity
gradients vanish in the corners and at the center of the domain (see~\Cref{fig:snapshots}). We prescribe
homogeneous data, \(\uu=\bm 0 \text{ on }\Omega\times \{0\}\),
\(\uu=\bm 0 \text{ on } \partial \Omega\times (0, T]\). The space-time mesh
$\Th\otimes\mathcal{M}_{\tau}$ is a uniform tensor-product mesh of
$\Omega\times \It$ and is refined uniformly in space and time. The main error quantity is
$\|\Phi_\reg(\DD\uu)-\Phi_\reg(\DD\uu_{\tau h})\|_{L^2((0,T);\,L^2(\Omega))}$,
i.e.\ the space-time $L^2(L^2)$-norm of
\begin{equation}\label{eq:nat-dist}
  \Phi_\reg(\bm A)\coloneq {(\reg^2+|\bm A|^2)}^{\frac{\pn-2}{4}}\bm A.
  \qquad \bm A\in\mathbb{R}^{d\times d}_{\mathrm{sym}}.
\end{equation}
\begin{table}[htb]
  \caption{Errors for
    $\mathbb{Q}_2^2/\mathbb{P}_1^{\disc}/\mathrm{DG}(1)$
    discretizations of the $(\pn,\reg)$-Navier-Stokes system
    with solution~\eqref{eq:conv-test} in the natural
    distance~\eqref{eq:nat-dist} and divergence of the velocity.}\label{tab:conv-pns}
  \begin{subcaptionblock}{\textwidth}
    \caption{Calculated errors in the space-time $L^2$-norm with eoc for $\nu_\infty=0$.}
    \setlength{\tabcolsep}{5pt} \centering
    \begin{tabular}{l|llll|llll}
      \toprule
      &\multicolumn{4}{c|}{$\pn=\tfrac32$, $\reg=10^{-15}$, $\nu=10^{-2}$, $\nu_{\infty}=0$}
      &\multicolumn{4}{c}{$\pn=\tfrac54$, $\reg=10^{-5}$, $\nu=10^{-2}$, $\nu_{\infty}=0$}\\[4pt]
      $h$ &
      $e^{\Phi_{\reg}(\DD\uu)}_{L^2(L^2)}$ &eoc&$e^{\boldsymbol\nabla\cdot \uu}_{L^2(L^2)}$&eoc&
      $e^{\Phi_{\reg}(\DD\uu)}_{L^2(L^2)}$ &eoc&$e^{\boldsymbol\nabla\cdot \uu}_{L^2(L^2)}$&eoc \\
      \midrule
      ${2}^{-2}$ &\num{4.87470e-01} &     &\num{2.9097e-01}&     &\num{4.48403e-01} &     &\num{2.6211e-01} &     \\
      ${2}^{-3} $&\num{1.90222e-01} &1.36 &\num{6.4618e-02}&2.17 &\num{1.78193e-01} &1.33 &\num{6.6096e-02} &1.99 \\
      ${2}^{-4} $&\num{8.41645e-02} &1.18 &\num{1.0454e-02}&2.63 &\num{7.59268e-02} &1.23 &\num{1.1344e-02} &2.54 \\
      ${2}^{-5}$& \num{4.00681e-02} &1.07 &\num{1.5706e-03}&2.73 &\num{3.37933e-02} &1.17 &\num{1.7076e-03} &2.73 \\
      ${2}^{-6}$& \num{1.92389e-02} &1.06 &\num{2.6846e-04}&2.55 &\num{1.49391e-02} &1.18 &\num{2.8323e-04} &2.59 \\
      ${2}^{-7}$& \num{8.33996e-03} &1.21 &\num{5.6162e-05}&2.26 &\num{6.60254e-03} &1.18 &\num{5.5898e-05} &2.34 \\
      \bottomrule
    \end{tabular}
  \end{subcaptionblock}
  \begin{subcaptionblock}{\textwidth}
    \caption{Calculated errors in the space-time $L^2$-norm with eoc for $\nu_\infty=10^{-5}$.}
    \setlength{\tabcolsep}{5pt} \centering
    \begin{tabular}{l|llll|llll}
      \toprule
      &\multicolumn{4}{c|}{$\pn=\tfrac32$, $\reg=10^{-15}$, $\nu=10^{-2}$, $\nu_{\infty}=10^{-5}$}
      &\multicolumn{4}{c}{$\pn=\tfrac54$, $\reg=10^{-5}$, $\nu=10^{-2}$, $\nu_{\infty}=10^{-5}$}\\[4pt]
      $h$ &
      $e^{\Phi_{\reg}(\DD\uu)}_{L^2(L^2)}$ &eoc&$e^{\boldsymbol\nabla\cdot \uu}_{L^2(L^2)}$&eoc&
      $e^{\Phi_{\reg}(\DD\uu)}_{L^2(L^2)}$ &eoc&$e^{\boldsymbol\nabla\cdot \uu}_{L^2(L^2)}$&eoc \\
      \midrule
      ${2}^{-2}$ &\num{4.87269e-01} &     &\num{2.9073e-01}&     &\num{4.48239e-01} &     &\num{2.6191e-01} &     \\
      ${2}^{-3} $&\num{1.90169e-01} &1.36 &\num{6.4554e-02}&2.17 &\num{1.78137e-01} &1.33 &\num{6.6026e-02} &1.99 \\
      ${2}^{-4} $&\num{8.41497e-02} &1.18 &\num{1.0443e-02}&2.63 &\num{7.59129e-02} &1.23 &\num{1.1330e-02} &2.54 \\
      ${2}^{-5}$& \num{4.00616e-02} &1.07 &\num{1.5692e-03}&2.73 &\num{3.37881e-02} &1.17 &\num{1.7057e-03} &2.73 \\
      ${2}^{-6}$& \num{1.92355e-02} &1.06 &\num{2.6832e-04}&2.55 &\num{1.49369e-02} &1.18 &\num{2.8304e-04} &2.59 \\
      ${2}^{-7}$& \num{8.33866e-03} &1.21 &\num{5.6155e-05}&2.26 &\num{6.59242e-03} &1.18 &\num{5.5903e-05} &2.34 \\
      \bottomrule
    \end{tabular}
  \end{subcaptionblock}
\end{table}

\paragraph{Verification}
To verify the implementation in the non-uniformly elliptic setting, we consider
$(\pn,\reg)=(1.5,10^{-15})$ and $(\pn,\reg)=(1.25,10^{-5})$ with $\nu_\infty=0$
or $\nu_{\infty}=10^{-5}$ and $\nu=10^{-2}$. For $\nu_\infty>0$, decreasing
$\nu_\infty$ increases the apparent Reynolds number in \eqref{eq:reapp}.
Table~\ref{tab:conv-pns} reports errors and experimental orders of convergence
(eoc) for the manufactured solution. The computations were performed with the modified Newton solver. In this regime, $(\pn,\reg)$-structure
theory does not predict a uniform second-order rate in the natural distance
without additional uniform-ellipticity and regularity assumptions. We therefore
use the table as a consistency check for the
$\mathbb{Q}_2^2/\mathbb{P}_1^{\disc}/\mathrm{DG}(1)$ discretization. Both the
natural-distance error $e^{\Phi_{\reg}(\DD\uu)}_{L^2(L^2)}$ and the divergence
of the velocity $e^{\nabla\cdot\uu}_{L^2(L^2)}$ decay under uniform space-time
refinement. Moreover, $e^{\nabla\cdot\uu}_{L^2(L^2)}$ decreases essentially
quadratically, whereas $e^{\Phi_{\reg}(\DD\uu)}_{L^2(L^2)}$ shows a
$\pn$-dependent subquadratic decay on the reported levels. At the reported
precision, the results for $\nu_{\infty}=0$ and $\nu_{\infty}=10^{-5}$ are
numerically indistinguishable. Therefore, in the following we focus on the $\nu_{\infty}=0$
case.

\paragraph{Solver evaluation}
\label{sec:crit-bench}
We assess solver robustness for parameter tuples
\(
  (\pn,\reg,\nu,\nu_\infty)\in\mathcal P,
\)
where
\begin{equation*}
\begin{split}
  \mathcal P
  &\coloneq
  \{1.16,1.2,1.25,1.33,1.5,1.66\}
  \times
  \{10^{-5},10^{-10},10^{-15},10^{-20}\}\\
  &\qquad\times
  \{10^{-1},10^{-2},10^{-3}\}
  \times
  \{10^{-5},0\}.
\end{split}
\end{equation*}
The full parameter set $\mathcal P$ is large. Thus, we compare Picard (Pic),
exact Newton (exN) and modified Newton (modN). For these nonlinear solvers, all other solver components, in
particular the tolerances, are kept fixed. The convergence test is run on meshes
ranging from 16 to \num[scientific-notation=false,round-precision=0]{262144}
cells, which corresponds to \num{2887682} spatial degrees of freedom (dof). We first inspect representative parameter subsets and then summarize the full benchmark set by a Dolan-Mor\'e performance profile~\cite{dolan_benchmarking_2002}.

\begin{table}[t]
\centering
\caption{Average nonlinear iteration count for $\nu=10^{-3}$ and $\reg=10^{-5}$. Variant~Pic deteriorates toward smaller $\pn$, whereas Variants~exN and modN remain $h$-robust.}
\label{tab:newton-compact}
\setlength{\tabcolsep}{4pt}
\begin{tabular}{r|rrr|rrr|rrr|rrr|rrr}
\toprule
& \multicolumn{3}{c|}{$\pn=1.50$}
& \multicolumn{3}{c|}{$\pn=1.33$}
& \multicolumn{3}{c|}{$\pn=1.25$}
& \multicolumn{3}{c|}{$\pn=1.20$}
& \multicolumn{3}{c}{$\pn=1.16$} \\[3pt]
\#cells
& \mc{c}{Pic} & \mc{c}{exN} & \mc{c|}{modN}
& \mc{c}{Pic} & \mc{c}{exN} & \mc{c|}{modN}
& \mc{c}{Pic} & \mc{c}{exN} & \mc{c|}{modN}
& \mc{c}{Pic} & \mc{c}{exN} & \mc{c|}{modN}
& \mc{c}{Pic} & \mc{c}{exN} & \mc{c}{modN} \\
\midrule
    16 & 4.00 &  8.12 & 8.00 &   7.38 &  8.75 & 7.12 &   7.62 & 10.88 & 7.75 &   7.00 & -- & 7.25 &   -- & -- & 7.38 \\
    64 & 4.44 &  8.81 & 7.62 &   7.94 & 11.12 & 7.81 &   7.75 & 12.88 & 7.15 &   7.75 & -- & 7.69 &   -- & -- & 7.69 \\
   256 & 3.03 &  8.06 & 5.78 &   6.38 & 14.69 & 6.19 &   6.38 & 22.56 & 5.62 &   6.31 & -- & 6.28 &   -- & -- & 6.28 \\
  1024 & 3.02 & 10.48 & 5.53 &   5.81 & 27.17 & 5.56 &   5.67 & 19.11 & 5.16 &   5.69 & -- & 5.69 &   -- & -- & 5.66 \\
  4096 & 3.12 & 13.25 & 5.26 &   5.39 & 21.84 & 5.19 &   5.16 & 15.38 & 5.06 &   5.10 & -- & 5.10 &   -- & -- & 5.11 \\
 16384 & 3.03 & 18.50 & 4.79 &   4.93 & 13.82 & 4.80 &   5.14 & 13.73 & 4.52 &   4.92 & -- & 4.92 &   -- & -- & 4.77 \\
 65536 & 3.05 & 24.57 & 4.58 &   4.86 & 11.20 & 4.80 &   5.09 & 12.77 & 4.79 &   5.27 & -- & 4.77 &   -- & -- & 4.96 \\
262144 & 4.02 & 28.48 & 4.16 &   5.55 & 10.05 & 5.00 &   5.61 & 11.94 & 5.07 &   5.53 & -- & 5.25 &   -- & -- & 7.21 \\
\bottomrule
\end{tabular}
\end{table}

Here $n_{\mathrm{NL}}(n)$ denotes the number of nonlinear iterations on $\Sn$~\eqref{eq:In-Sn}, and $n_{\mathrm{L}}(n,m)$ the corresponding
FGMRES iteration count in nonlinear iteration $m$.
As aggregate cost measure we use the work
\begin{equation}\label{eq:work}
  W \coloneq \Bigl(\sum_{n}\sum_{m=1}^{n_{\mathrm{NL}}(n)} n_{\mathrm{L}}(n,m)\Bigr)\,N_{\mathrm{dof}},
\end{equation}
that is, the total number of preconditioned Krylov iterations multiplied by $N_{\mathrm{dof}}$, the number of space-time degrees of freedom per time step. The arithmetic mean of the
nonlinear iteration count over all time steps is denoted by $\overline{n}_{NL}$.

To expose the nonlinear solver behavior directly, \Cref{tab:newton-compact}
reports the mean nonlinear iteration count $\overline{n}_{\mathrm{NL}}$ over all
time steps for the representative regime $\nu=10^{-3}$ and $\reg=10^{-5}$,
resolved by the number of spatial cells, the power-law exponent $\pn$, and the
nonlinear solver. The table shows a clear separation between the methods. Picard
remains usable for moderate shear-thinning, but its performance deteriorates as
$\pn$ decreases and it fails for $\pn=1.16$. Exact Newton deteriorates more
severely: its iteration counts increase strongly under refinement already for
$\pn=1.33$ and $\pn=1.25$, and it breaks down for $\pn\le 1.20$. By contrast,
the modified Newton solver remains stable and essentially $h$-robust throughout,
typically requiring about five to eight nonlinear iterations. Its advantage is
most pronounced in the strongly shear-thinning regime, where it is the only
variant that solves all cases in the table.

To compare the tangent variants on the full benchmark set, we use a Dolan-Mor\'e
performance profile with the work~\eqref{eq:work} as cost.
Each benchmark instance $i\in\mathcal P$ corresponds to fixed model parameters and a fixed space-time resolution, and each curve
corresponds to one solver variant $s\in\{\text{Pic},\,\text{exN},\,\text{modN}\}$. For every instance we
normalize the work by the best variant on that instance and define the
performance profile $\pi_s(\tau)$ of variant $s$
\begin{equation}\label{eq:perf-profile}
  r_{i,s}\coloneq \frac{W_{i,s}}{\min_{s'\in\mathcal S}
    W_{i,s'}}\in[1,\infty)\,,\quad \pi_s(\tau)\coloneq \frac{1}{|\mathcal P|}\#\{i\in\mathcal P:\ r_{i,s}\le \tau\},\quad \tau\ge 1\,.
\end{equation}
Thus $\pi_s(\tau)$ is the fraction of benchmark instances for which solver $s$ is within a factor
$\tau$ of the smallest observed work~\eqref{eq:work}. In particular, $\pi_s(1)$ is the fraction of instances on which $s$ attains the smallest work; larger $\pi_s(\tau)$ indicates better overall efficiency across the benchmark set. \Cref{fig:dm-work} reports
$\pi_s(\tau)$ for all variants.
\begin{figure}
\centering
\includegraphics[height=4cm]{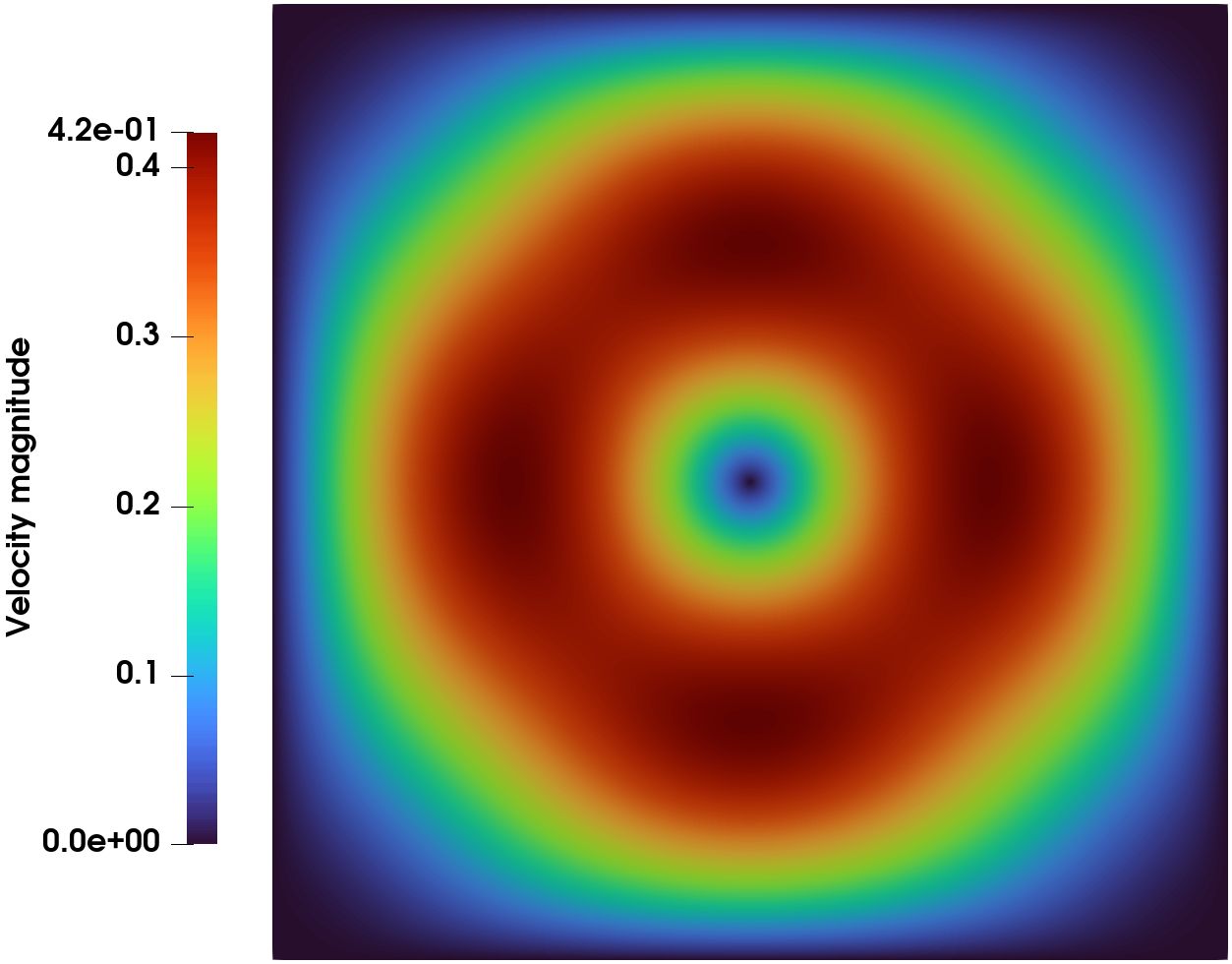}\hfill
\includegraphics[height=4cm]{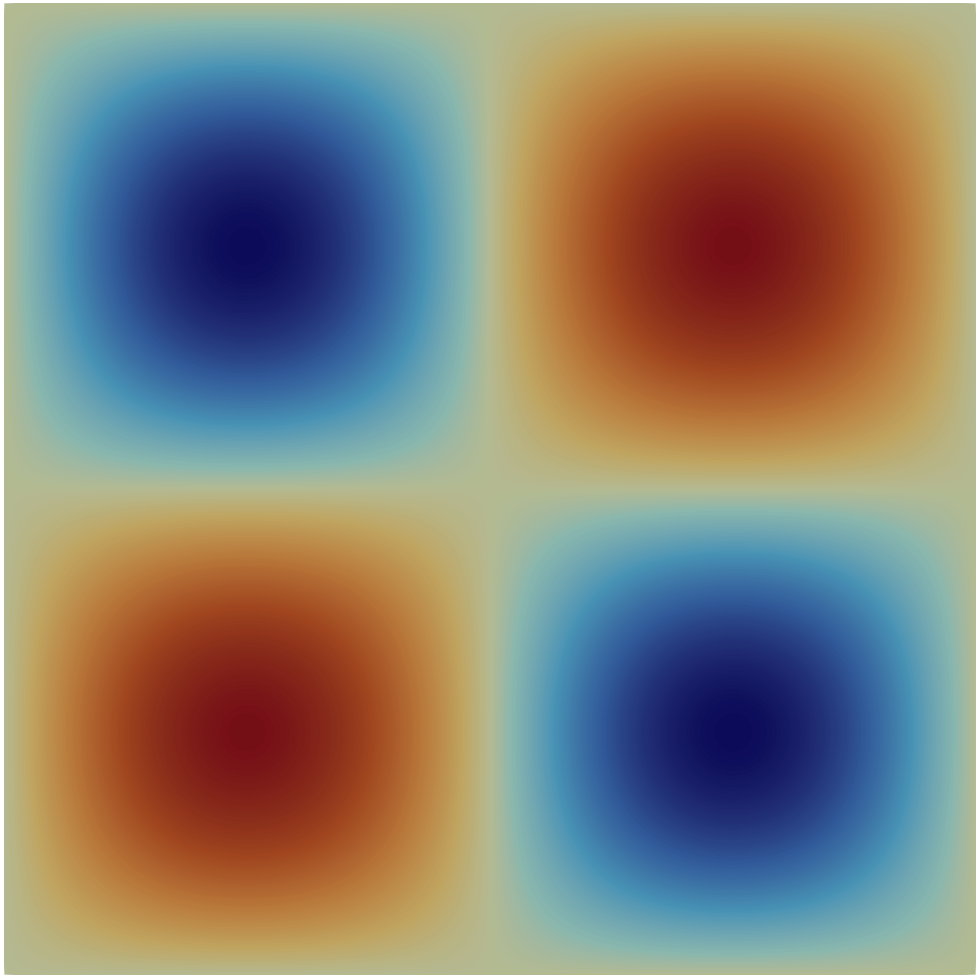}\hfill
\includegraphics[height=4cm]{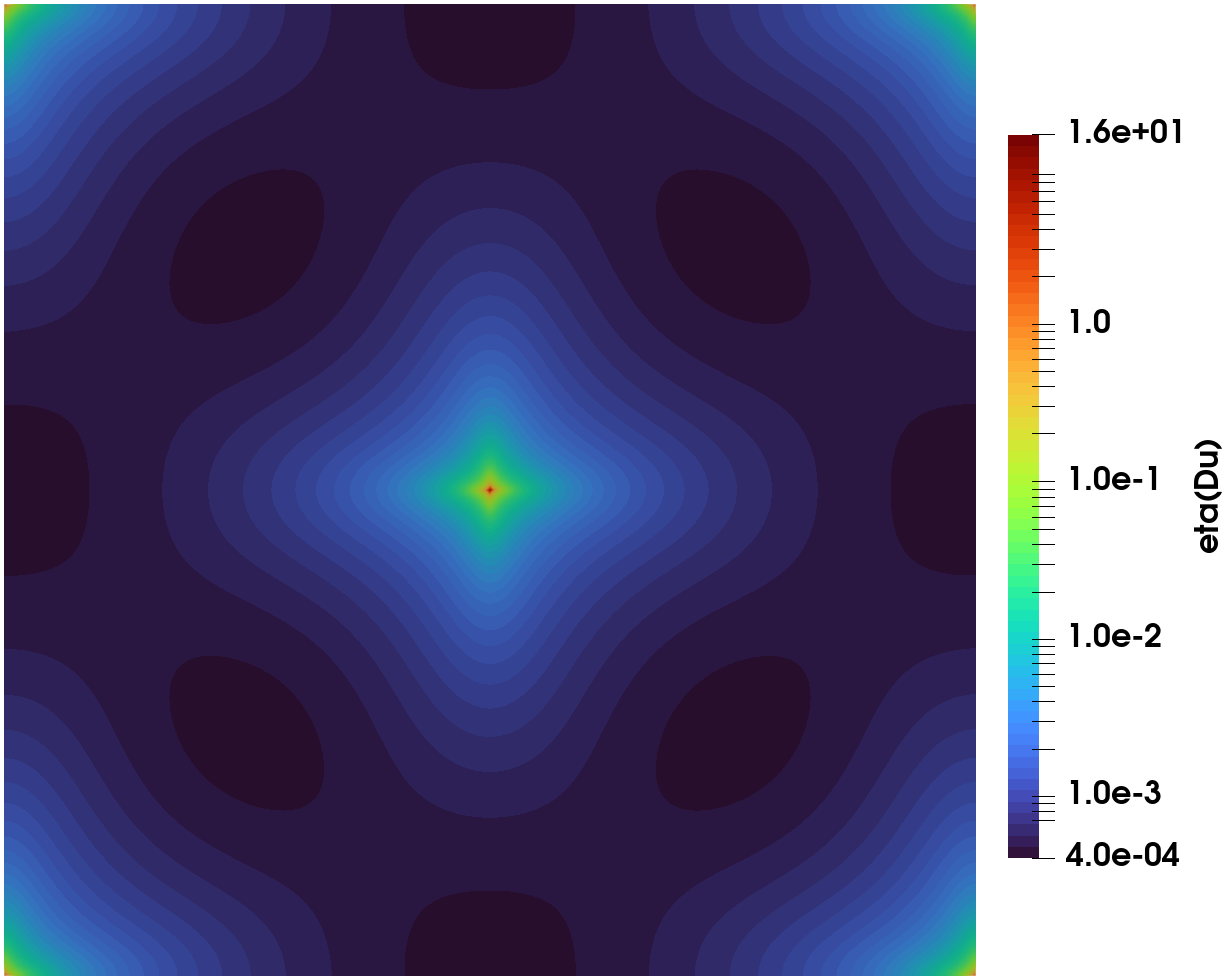}
\caption{\label{fig:snapshots}Plots of the calculated velocity, pressure and
  viscosity for the convergence test with $(\pn,\reg)=(1.16,10^{-5})$ and
  $\nu=10^{-3}$ on a mesh with
  \num[scientific-notation=false,round-precision=0]{262144} cells and 1024 time
  steps at the final time.}
\end{figure}
\begin{figure}
  \includegraphics{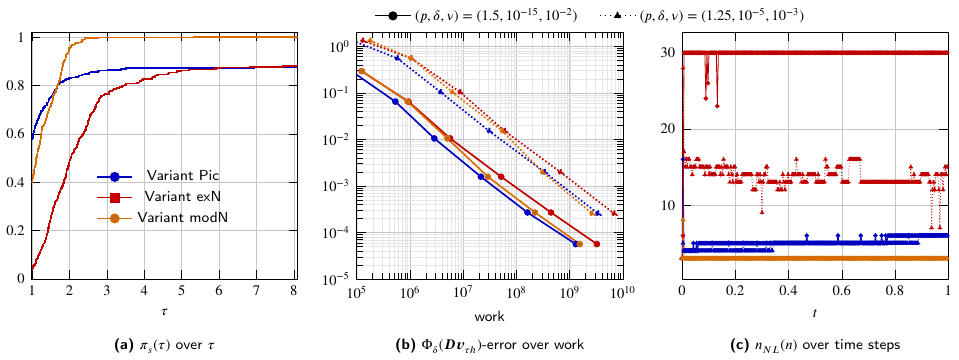}

\caption{\label{fig:dm-work}Performance summary. (a) Dolan-Mor\'e profile $\pi_s(\tau)$ over the full grid
  (work factor $\tau$ w.r.t.\ the per-instance minimum). (b) $\Phi_\reg(\DD\uu_{\tau h})$-error versus work,
  $\|\Phi_\reg(\DD\uu)-\Phi_\reg(\DD\uu_{\tau h})\|_{L^2((0,T);\,L^2(\Omega))}$, on representative tuples.
  (c) nonlinear iteration histories $n_{NL}(n)$ for the same tuples.}
\end{figure}
\begin{figure}[htbp]
  \centering
  \includegraphics{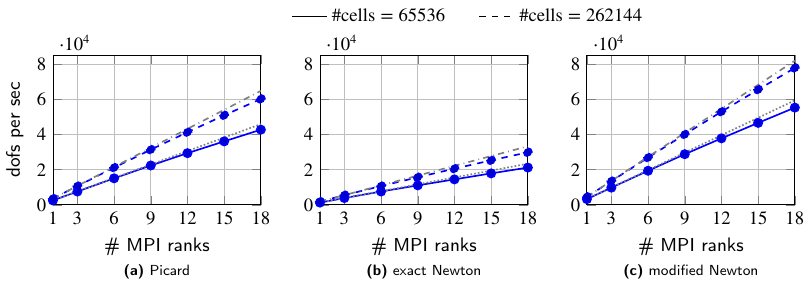}

\caption{\label{fig:strong-scaling-dofs}
Strong scaling results for the monolithic solver in degrees of freedom per
second for the Picard, exact Newton, and modified Newton linearizations on two
successive refinement levels.}
\end{figure}
\Cref{fig:dm-work} summarizes efficiency, robustness, and nonlinear iteration
behavior of the nonlinear solvers. In panel (a), the Dolan-Mor\'e profile of a
nonlinear solver reports the fraction of problems solved within a factor $\tau$
of the best performing solver. The modified Newton method is robust in almost
all settings. The only outliers are the least regularized, most strongly
shear-thinning cases \((\pn,\reg)=(1.16,10^{-20})\), where none of the variants
succeeds. Picard and exact Newton fail in substantially more cases and solve
fewer than \(90\%\) of the instances. For coarse meshes and moderate
\((\pn,\reg)\), Picard remains competitive and can even be the fastest option.
However, this advantage is limited: whenever modified Newton is not the fastest
variant, its work remains within a factor of two of the best observed
performance, while its reliability is substantially higher. Overall, modified
Newton provides the most favorable balance between reliability and efficiency
across the parameter regimes considered, in particular on fine meshes and for
\(\pn\downarrow 1\) and \(\delta\downarrow 0\). Panel~(b) shows the asymptotic
error-work slopes for the different solvers. The slopes are similar, so the
differences appear mainly as a horizontal shift: modified Newton and Picard
require less work than exact Newton to reach the same error. On coarse meshes,
Picard can be slightly more efficient, but this advantage is lost under
refinement. The iteration histories in panel~(c) show the same ordering. Exact
Newton exhibits pronounced spikes and persistently high iteration counts,
whereas the modified Newton solvers remain low and nearly uniform in time.
Picard lies between exact and modified Newton.

\Cref{fig:strong-scaling-dofs} shows near-ideal strong scaling for all three
linearizations on both refinement levels for $(\pn,\reg)=(1.25,10^{-5})$ and
$\nu=10^{-3}$. This is consistent with the fact that the dominant cost lies in
the local Vanka smoother, whose work is largely embarrassingly parallel. The
modified Newton variant attains the highest throughput, followed by Picard and
exact Newton. Picard outperforms exact Newton because the linear solver
converges faster, while the number of Picard iterations remains moderate,
leading to a lower overall computational cost. The finer level performs slightly
better in all three cases, indicating that the parallel overhead is effectively
amortized as the local problem size increases.

\subsection{Time-dependent flow around a cylinder: DFG benchmark}
\label{sec:dfg-bench}
\begin{figure}[!htb]
  \centering
  \includegraphics{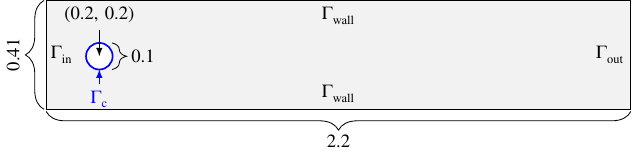}
\caption{Geometry of the test scenario with a parabolic inflow profile $\Gamma_{\textrm{in}}$,
  do-nothing boundary conditions at the outflow boundary $\Gamma_{\textrm{out}}$ and no-slip conditions on the obstacle and walls
  $\Gamma_{\textrm{wall}}$. The center of the obstacle is at $(0.2,\,0.2)$.}\label{fig:geom}
\end{figure}
\begin{figure}
  \centering
  \includegraphics[width=0.8\linewidth]{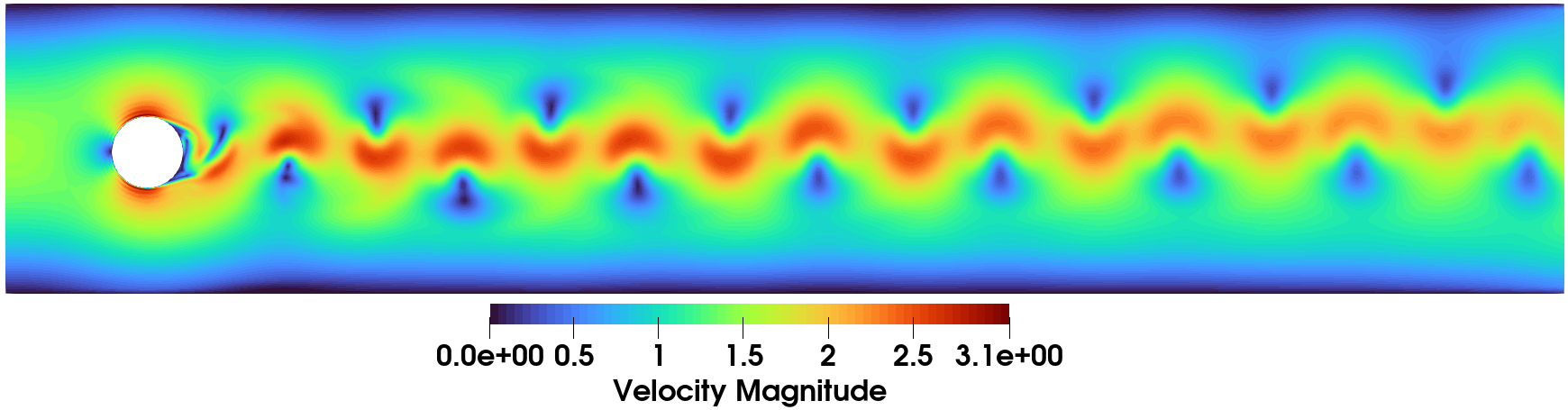}\hfill
  \includegraphics[width=0.8\linewidth]{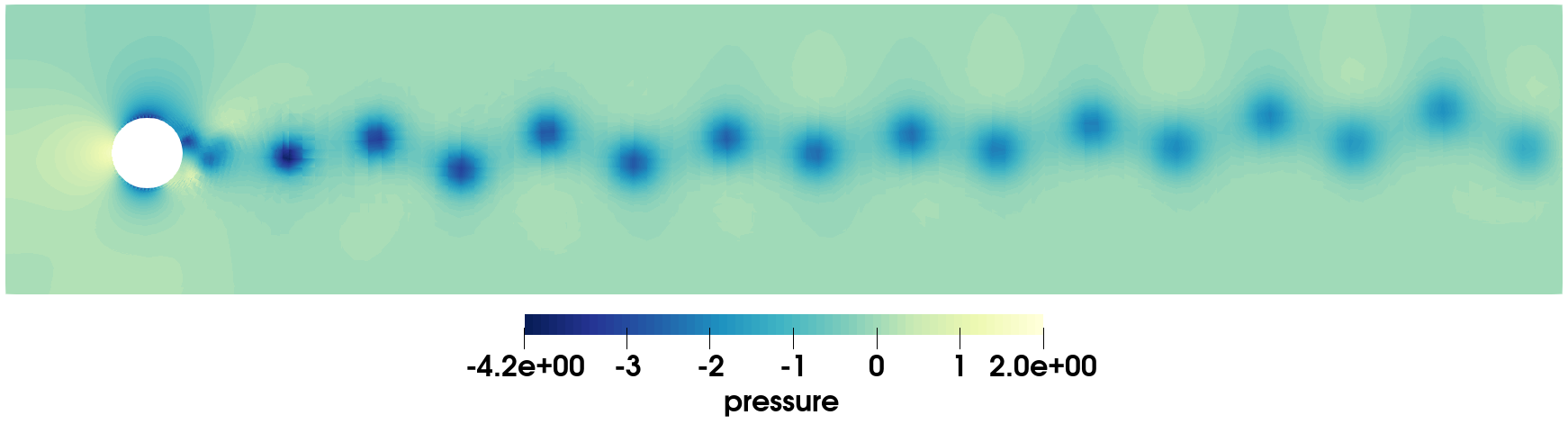}\hfill
  \includegraphics[width=0.8\linewidth]{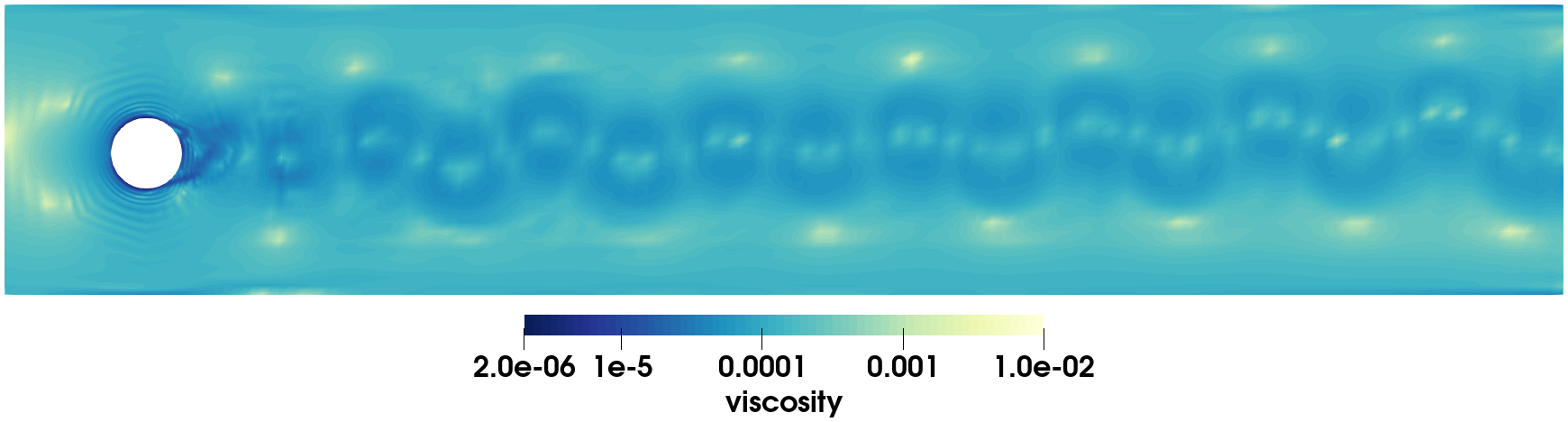}
  \caption{\label{fig:dfg-snapshots}Plots of the calculated velocity, pressure and
    viscosity for the DFG benchmark with $(\pn,\reg)=(1.25,10^{-10})$ and
    $\nu=10^{-3}$ on a mesh with
    \num[scientific-notation=false,round-precision=0]{4864} cells at the final
    time $T=8$.}
\end{figure}
\begin{figure}
  \centering
  \includegraphics{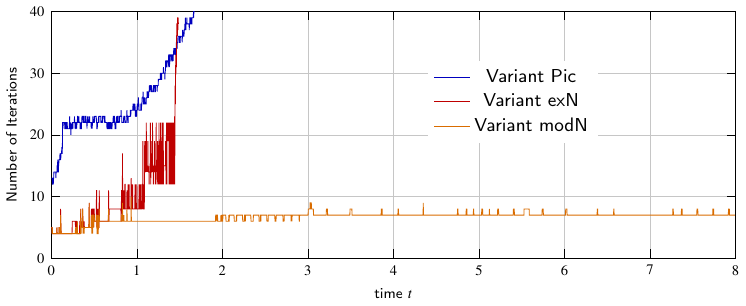}
  \caption{\label{fig:dfg-its}Plot of the nonlinear iteration histories $n_{NL}(n)$ for the Picard,
    solver, modified and exact Newton solver at spatial cells 4864 and 3392 timesteps. The exact Newton iteration stagnates, whereas the Picard iteration is too slow to reduce the residual sufficiently.}
\end{figure}
\begin{table}[htb]
  \centering
  \caption{The mean and maximum numbers of nonlinear iterations and FGMRES iterations per time step, reported over all time steps, for varying spatial and temporal resolutions (numbers of cells and degrees of freedom in space and time).}
  \label{tab:st-sizes}
  \sisetup{scientific-notation=true, round-precision=2}
  \begin{tabular}{
  S[scientific-notation=false,round-precision=0,table-format=5]
  S[scientific-notation=false,round-precision=0,table-format=5]
  S[scientific-notation=false,round-precision=0,table-format=6]
  S[scientific-notation=false,round-precision=0,table-format=5]
  S[scientific-notation=false,round-precision=2]
  S[scientific-notation=false,round-precision=2]
  S[scientific-notation=false,round-precision=0]
  S[scientific-notation=false,round-precision=0]
  }
    \toprule
     \multicolumn{2}{c}{\# elements}
     & \multicolumn{2}{c}{\# dofs}
     & \multicolumn{4}{c}{\# iterations}
     \\
    \mc{c}{space}
    & \mc{c}{time}
    & \mc{c}{$\uu$}
    & \mc{c}{$\pi$}
    & \mc{c}{$\overline{n}_{NL}$}
    & \mc{c}{$\overline{n}_{L}$}
    & \mc{c}{$\max {n}_{NL}$}
    & \mc{c}{$\max {n}_{L}$}\\
    \midrule
    304   & 848  &2624   &912   & 7.98 & 7.63 & 9 & 22\\
    1216  & 1696 &10112  &3648  & 7.49 & 7.41 & 9 & 28\\
    4864  & 3392 &39680  &14592 & 6.51 & 22.74 & 9 & 39\\
    19456 & 6784 &157184 &58368 & 6.02 & 23.62 & 9 & 47\\
    \bottomrule
  \end{tabular}
\end{table}
\begin{figure}
\begin{centering}
  \includegraphics{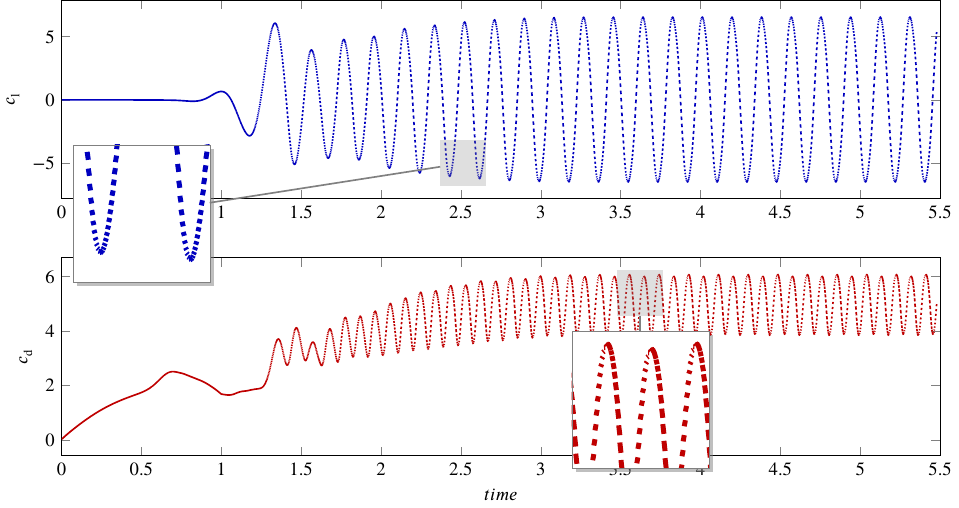}
\end{centering}
\caption{\label{fig:dl}Drag and lift functionals~\eqref{eq:drag-lift} plotted over time. We
  plot the functionals with the actual discontinuities between the time steps.}
\end{figure}

We consider a variation of the well-established ``laminar flow around a
cylinder'' benchmark of Sch\"afer and Turek~\cite{SchaeferTurek1996}, in
particular the unsteady 2D-2 testcase. This time-dependent example probes the time-step-wise
behavior of the fully implicit space-time solver, rather than only its asymptotic refinement behavior.
At the left boundary $\Gamma_{\mathrm{in}}\coloneq
0\times [0,\,H]$ we impose the Dirichlet inflow profile $\uu=\uu^D$ with
\begin{equation}
\uu^D(y,t) = \begin{pmatrix}
\bar{v}\dfrac{6y(H-y)}{H^2}\omega(t)\\[0.3em]
0
\end{pmatrix},
\quad \omega(t) = \begin{cases} \frac{1}{2}- \frac{1}{2}\cos( \pi t), & t\le 1, \\ 1, & t>1,
\end{cases}
\end{equation}
where $H=0.41$ is the channel height, $\bar v$ is the mean inflow velocity, and
$\omega(t)$ regularizes the startup phase of the flow. On the wall boundary
$\Gamma_{\mathrm{wall}}$ and on the obstacle, no-slip boundary conditions
$\uu=0$ are prescribed. On the outflow boundary $\Gamma_{\mathrm{out}}$ a
do-nothing outflow condition~\cite{HeywoodRannacherTurek1992} is used.
We evaluate the drag $c_{\mathrm d}$ and lift $c_{\mathrm l}$ acting on the obstacle through the functionals
\begin{equation}\label{eq:drag-lift}
  J_{i}(\uu,\pi)
  \coloneq \int_{\Gamma_{\mathrm c}}
  \bigl(\Sstress(\DD\uu)\nn-\pi\nn\bigr)\cdot\bm e_i \, \drv s,\; i\in\{1,\,2\},\qquad
  c_{\mathrm{d}}=\frac{2}{\bar v^{2} L}J_1,\;c_{\mathrm{l}}=\frac{2}{\bar v^{2} L}J_2
\end{equation}
where $L=0.1$ is the characteristic length, $\Gamma_{\mathrm c}$ denotes the
obstacle boundary, $\nn$ its outward unit normal, and $\bm e_i$ the Cartesian
unit vectors in the $i$-th direction. We fix $(\pn,\reg)=(1.25,10^{-10})$,
$\nu=10^{-3}$ and $\nu_\infty=0$. In this regime, convection stabilization is
required because of strong boundary layers at the obstacle. We use CIP with
$\gCIP=1$. Dirichlet conditions are imposed by Nitsche with penalties
$\gOne=\gTwo=10^3$, which we found necessary for robustness under strongly
varying viscosity. We employ the exact and modified Newton and the Picard
iteration. In \Cref{fig:dfg-its}, the exact Newton iteration stagnates as soon
as the temporal dynamics develop over only a few time steps. The Picard
iteration shows the same qualitative behavior: although it still advances the
computation, it repeatedly reaches the prescribed iteration limit before
achieving sufficient residual reduction. We therefore restrict the following
discussion to the modified Newton method.

Figure~\ref{fig:dfg-snapshots} shows a representative final-time solution. The wake
is well developed, the pressure field resolves the alternating vortex pattern,
and the effective viscosity is strongly reduced near the obstacle and in the
wake. These plots indicate that the discretization captures the main
shear-thinning flow features for this parameter set. Table~\ref{tab:st-sizes} shows that the time-dependent nonlinear solve remains robust across the full simulation horizon. The mean nonlinear iteration counts stay moderate and decrease under refinement, while the maximal count remains \(9\) on all levels. Thus the outer nonlinear iteration is stable not only under mesh refinement but also over many successive time steps. The linear solver shows a different behavior: the mean FGMRES counts increase strongly from the coarse to the fine meshes, but then stabilize on the two finest levels, whereas the maximal counts continue to increase. This identifies the linear algebra, rather than the nonlinear iteration, as the dominant remaining bottleneck for large time-dependent runs. Overall, the results show that the fully implicit space-time formulation can be advanced reliably in this challenging regime, with stable nonlinear performance and controlled linear-solver growth. Figure~\ref{fig:dl} shows representative drag and lift histories for a fine space-time discretization. The discontinuities between time steps are visible, as expected for a DG-in-time solution representation, while the overall histories remain stable and physically plausible throughout the simulation.

\FloatBarrier
\section{Conclusions and outlook}
\label{sec:conclusions}
We have presented a scalable Newton-Krylov framework for fully implicit
tensor-product space-time finite element discretizations of incompressible
$(\pn,\reg)$-Navier-Stokes flow in the shear-thinning regime $1<\pn<2$. The
decisive algorithmic difficulty is the constitutive tangent: as $\pn\downarrow
1$ and $\reg\downarrow 0$, the exact tangent becomes increasingly
ill-conditioned, which impairs both Newton globalization and preconditioning.
Accordingly, the key design choice is its treatment in the nonlinear
linearization. In the parameter regimes considered here, exact Newton
stagnates and Picard converges too slowly, whereas the modified Newton method
is the only variant that remains reliable. It keeps the nonlinear residual
unchanged and replaces only the constitutive tangent in the Jacobian action by
a better-conditioned surrogate.

Numerically, we recover the expected convergence rates for a manufactured
solution and observe bounded nonlinear iteration counts under $h$-refinement for
the modified Newton solver. The method remains effective as $\pn$ decreases and
also in the limit $\nu_\infty=0$. For the time-dependent DFG benchmark in
a strongly shear-thinning regime, the modified Newton solver shows robust
time-step-wise nonlinear performance over the full time interval. The main
remaining limitation lies in the linear solve: the iteration counts are not yet
fully mesh-robust, although they remain controlled on the finest levels
considered.

Taken together, these results show that fully implicit monolithic space-time
solvers for strongly shear-thinning flow are viable, provided that the
constitutive tangent is treated with care. The present framework provides a
solid basis for further development toward more demanding regimes and
large-scale simulations. Natural next steps include adaptive
$(h,\tau)$ refinement with robust patch rebuild criteria, and large-scale
studies of smoothing effort, rebuild frequency, and time-to-solution in
convection-dominated regimes.

\paragraph*{Acknowledgements.}
This project has received funding from the European Regional Development Fund (grant FEM Power II, ZS2024/06/18815) under the European Union's Horizon Europe Research and Innovation Program, which we gratefully acknowledge. Carolin Mehlmann acknowledges funding by the Deutsche Forschungsgemeinschaft (DFG, German Research Foundation) (SPP 1158: project number 463061012).

\appendix
\section{Gauss-Radau quadrature and remainder estimates}\label{sec:gr-underint}
On each $\In$ we evaluate time integrals of the spatial terms by the right-sided $(k{+}1)$-point
Gauss-Radau rule $Q_n$. It preserves the endpoint $t_n$ and is exact on $\mathbb P_{2k}(\In)$
(cf.~\cite{DavisRabinowitz1984Methods}). This choice yields the
local time matrices with the tensor-product structure used by the matrix-free implementation and
the monolithic space-time multigrid preconditioner. The following bounds show that the induced
quadrature defect is order-preserving.

\begin{lemma}[Gauss-Radau remainder]\label{lem:gr-quadrature}
Let $Q_n$ be the $(k{+}1)$-point Gauss-Radau quadrature rule as above.
If $\phi\in W^{2k+1,\infty}(\In)$, then with $C_k>0$ depending on $k$,
\[
  \Bigl|\int_{\In}\phi(t)\Dt - Q_n(\phi)\Bigr|
  \le C_k\,\tau_n^{2k+2}\,\|\phi^{(2k+1)}\|_{L^\infty(\In)}\,.
\]
\end{lemma}

\begin{lemma}[Smoothness and higher derivatives]
  \label{lem:stress-higher-derivatives}
  Assume $\reg>0$. Then $\Sstress\in C^\infty(\mathbb{R}^{d\times d}_{\mathrm{sym}};\mathbb{R}^{d\times d}_{\mathrm{sym}})$.
  Moreover, for every integer $m\ge 1$ there exists a constant $C_m>0$, depending only on $m$, $\pn$, and $d$, such that for all $\bm A\in\mathbb{R}^{d\times d}_{\mathrm{sym}}$,
  \[
    \norm{D^m\Sstress(\bm A)}
    \ \le\ C_m\Bigl(\nu_{\infty}\,\mathbf 1_{\{m=1\}}+\nu\,\reg^{\pn-1-m}\Bigr).
  \]
\end{lemma}
\begin{proof}
  Write
  \[
    \Sstress(\bm A)
    = \nu_\infty \bm A + \nu\,\psi(\bm A)\,\bm A,
    \qquad
    \psi(\bm A)\coloneq (\reg^2+\abs{\bm A}^2)^{\frac{\pn-2}{2}}.
  \]
  Since $\reg>0$, the map $\psi$ is $C^\infty$ on $\mathbb{R}^{d\times d}_{\sym}$. Repeated application of the chain rule to $\psi(\bm A)\bm A$ shows that each $m$-th derivative is bounded by a constant multiple of
  $\nu\,\reg^{\pn-1-m}$, while the linear part contributes only for $m=1$.
\end{proof}

\begin{lemma}[Quadrature error for convection and regularized stress]\label{lem:gr-conv-visc}
Assume $\reg>0$ and let $\uu_{\tau h},\vv_{\tau h}\in \Pk(\In;\Vh)$. Define
\[
  \phi_{\mathrm{conv}}(t)\coloneq \bigl(\uu_{\tau h}(t)\otimes\uu_{\tau h}(t),\GG\vv_{\tau h}(t)\bigr)_{L^2(\Omega)},
  \quad
  \phi_{\mathrm{visc}}(t)\coloneq \bigl(\Sstress(\DD\uu_{\tau h}(t)),\DD\vv_{\tau h}(t)\bigr)_{L^2(\Omega)}.
\]
Then $\phi_{\mathrm{conv}},\phi_{\mathrm{visc}}\in W^{2k+1,\infty}(\In)$ and
\begin{equation}\label{eq:gr-conv-visc}
  \Bigl|\int_{\In}\phi_{\star}(t)\Dt - Q_n(\phi_{\star})\Bigr|
  \le C_k\,\tau_n^{2k+2}\,\|\phi_{\star}^{(2k+1)}\|_{L^\infty(\In)},
  \qquad \star\in\{\mathrm{conv},\mathrm{visc}\}.
\end{equation}
If $k=0$, then $\uu_{\tau h}$ and $\vv_{\tau h}$ are piecewise constant in
time on $\In$, hence $\phi_{\mathrm{conv}}$ and $\phi_{\mathrm{visc}}$ are
constant and $Q_n$ is exact. If $k\ge 1$, then there exists $C=C(k,d)$ such that
\begin{equation}\label{eq:gr-bound-conv}
  \begin{split}
    \|\phi_{\mathrm{conv}}^{(2k+1)}\|_{L^\infty(\In)}
    &\le C  \sum_{\substack{\alpha_1{+}\alpha_2{+}\alpha_3=2k{+}1\\0\le \alpha_i\le k}}\;
    \|\dt^{\alpha_3}\vv_{\tau h}\|_{L^\infty(\In;\boldsymbol H^1)}
    \prod_{i=1}^2\|\dt^{\alpha_i}\uu_{\tau h}\|_{L^\infty(\In;\boldsymbol H^1)},
  \end{split}
\end{equation}
and
\begin{equation}\label{eq:gr-bound-visc}
  \|\phi_{\mathrm{visc}}^{(2k+1)}\|_{L^\infty(\In)}
  \le C\sum_{j=0}^k
  \|\dt^{j}\DD\vv_{\tau h}\|_{L^\infty(\In;L^2)}
  \,
  \|\dt^{2k+1-j}\Sstress(\DD\uu_{\tau h})\|_{L^\infty(\In;L^2)}.
\end{equation}
The last factor in \eqref{eq:gr-bound-visc} is estimated by repeated chain rules: for fixed $\reg>0$,
\Cref{lem:stress-higher-derivatives} bounds $D^m\Sstress$, and hence \eqref{eq:gr-bound-visc} is finite.
\end{lemma}
\begin{proof}
Estimate~\eqref{eq:gr-conv-visc} is \Cref{lem:gr-quadrature}. For $k=0$ the
integrands are constant on $\In$, so $Q_n$ is exact. Let therefore $k\ge 1$.
Then the bounds \eqref{eq:gr-bound-conv}-\eqref{eq:gr-bound-visc} follow by
differentiating $\phi_{\mathrm{conv}}$, $\phi_{\mathrm{visc}}$ $2k{+}1$ times
and applying the product rule. For $\phi_{\mathrm{conv}}$ this is
\cite[Lemma~3.2]{MargenbergBause2026MonolithicSTMG}. All time derivatives
of $\uu_{\tau h}$ and $\vv_{\tau h}$ above order $k$ vanish. For
$\phi_{\mathrm{visc}}$ one applies the chain rule to
$\Sstress(\DD\uu_{\tau h})$, using boundedness of $D^m\Sstress$ for fixed
$\reg>0$ from \Cref{lem:stress-higher-derivatives}.
\end{proof}

\begin{lemma}[Assembled Gauss-Radau quadrature error on $\It$]\label{lem:gr-assembled}
Let $k\in\mathbb N_0$ and let $\uu_{\tau h}\in Y_\tau^k(\It)\otimes\Vh$ satisfy
\begin{equation}\label{eq:assump-time-stab}
\max_{0\le \ell\le k}\max_{n=1,\dots,N}
\|\dt^\ell\uu_{\tau h}\|_{L^\infty(\In;\boldsymbol W^{1,\infty}(\Omega))}
\le A_k.
\end{equation}
For $\vv_{\tau h}\in Y_\tau^k(\It)\otimes\Vh$ define on $\In$
\[
  \mathcal C_n(\uu_{\tau h})(\vv_{\tau h})\coloneq\int_{\In}\phi_{\mathrm{conv}}(t)\Dt,
  \qquad
  \mathcal V_n(\uu_{\tau h})(\vv_{\tau h})\coloneq\int_{\In}\phi_{\mathrm{visc}}(t)\Dt,
\]
and $\mathcal C_n^{\mathrm{GR}}\coloneq Q_n(\phi_{\mathrm{conv}})$,
$\mathcal V_n^{\mathrm{GR}}\coloneq Q_n(\phi_{\mathrm{visc}})$.
Then there exists $B_k=B_k(A_k,\pn,\nu_\infty,\reg)$ such that
\begin{equation}\label{eq:gr-assembled-bound}
\Bigl|\sum_{n=1}^N\bigl(\mathcal C_n-\mathcal C_n^{\mathrm{GR}}\bigr)\Bigr|
+\Bigl|\sum_{n=1}^N\bigl(\mathcal V_n-\mathcal V_n^{\mathrm{GR}}\bigr)\Bigr|
\le B_k\,\tau^{k+1}\,\|\vv_{\tau h}\|_{L^2(\It;\boldsymbol H^1(\Omega))}.
\end{equation}
\end{lemma}
\begin{proof}
By \Cref{lem:gr-quadrature,lem:gr-conv-visc},
\[
  |\mathcal C_n-\mathcal C_n^{\mathrm{GR}}|+|\mathcal V_n-\mathcal V_n^{\mathrm{GR}}|
  \lesssim \tau_n^{2k+2}\Bigl(\|\phi_{\mathrm{conv}}^{(2k+1)}\|_{L^\infty(\In)}+\|\phi_{\mathrm{visc}}^{(2k+1)}\|_{L^\infty(\In)}\Bigr).
\]
Using \eqref{eq:gr-bound-conv} and \eqref{eq:assump-time-stab} we obtain
$\|\phi_{\mathrm{conv}}^{(2k+1)}\|_{L^\infty(\In)}\lesssim A_k^2
\sum_{\ell=0}^k\|\dt^\ell\vv_{\tau h}\|_{L^\infty(\In;\boldsymbol H^1)}$.
For $\phi_{\mathrm{visc}}$, \eqref{eq:gr-bound-visc} and \Cref{lem:stress-higher-derivatives} yield the analogous bound
\[
\|\phi_{\mathrm{visc}}^{(2k+1)}\|_{L^\infty(\In)}
\lesssim \mathcal P_k(A_k)\sum_{\ell=0}^k\|\dt^\ell\vv_{\tau h}\|_{L^\infty(\In;\boldsymbol H^1)}.
\]
Using the polynomial inverse inequality in time~\cite[(3.28)]{KarakashianMakridakis1998}, we obtain
\[
\sum_{\ell=0}^k\|\dt^\ell\vv_{\tau h}\|_{L^\infty(\In;\boldsymbol H^1)}
\lesssim \tau_n^{-k-\frac12}\|\vv_{\tau h}\|_{L^2(\In;\boldsymbol H^1)}.
\]
Hence
\[
|\mathcal C_n-\mathcal C_n^{\mathrm{GR}}|+|\mathcal V_n-\mathcal V_n^{\mathrm{GR}}|
\lesssim B_k\,\tau_n^{k+\frac32}\|\vv_{\tau h}\|_{L^2(\In;\boldsymbol H^1)}.
\]
Summing over $n$ and using Cauchy-Schwarz yields~\eqref{eq:gr-assembled-bound}.
\end{proof}

\begin{remark}[Use in an error identity]\label{rem:gr-absorb}
Once $\vv_{\tau h}$ is chosen as an error quantity, the factor
$\|\vv_{\tau h}\|_{L^2(\It;\boldsymbol H^1)}$ in \eqref{eq:gr-assembled-bound}
can be absorbed by the viscous term after Cauchy-Young. Here we use
\Cref{lem:gr-assembled} only to justify that the quadrature defect is order-preserving in $\tau$
and does not alter the algebraic structure in \Cref{sec:algebraic}.
\end{remark}

\section{Spectrum of the exact \texorpdfstring{$(\pn,\reg)$}{(p,delta)} constitutive tangent}\label{app:tangent-spectrum}
Recall the $(\pn,\reg)$ stress law \eqref{eq:constitutive-law}. Its Fr\'echet derivative
(cf.\ \eqref{eq:pns-dstress}) reads, for $\bm A,\bm H\in\mathbb{R}^{d\times d}_{\sym}$,
\begin{equation}\label{eq:app_exact_tangent}
  \DStress{\bm A}{\bm H}
  =\etaeff{\bm A}\,\bm H
  +\beta(\bm A)\,(\bm A:\bm H)\,\bm A,
  \qquad
  \beta(\bm A)\coloneq \nu(\pn-2)\bigl(\reg^2+\abs{\bm A}^2\bigr)^{\frac{\pn-4}{2}}.
\end{equation}
Thus $\bm H\mapsto\DStress{\bm A}{\bm H}$ is a rank-one perturbation of
$\etaeff{\bm A}\Id$ on $\mathbb{R}^{d\times d}_{\sym}$, and the spectrum follows from
the next standard observation.

\begin{lemma}[Rank-one perturbations]\label{lem:rankone-tangent}
Let $\eta>0$, $\bm a\in\mathbb{R}^{d\times d}_{\sym}$, and $\kappa\in\mathbb{R}$, and define
$\mathcal T(\bm B)\coloneq \eta\,\bm B+\kappa(\bm a:\bm B)\,\bm a$ for $\bm B\in\mathbb{R}^{d\times d}_{\sym}$.
Then
\[
  (\mathcal T(\bm B):\bm B)=\eta\,\abs{\bm B}^2+\kappa\,(\bm a:\bm B)^2,
  \qquad
  \operatorname{spectrum}(\mathcal T)=\{\eta,\ \eta+\kappa\abs{\bm a}^2\},
\]
where $\eta$ acts on $\{\bm B:\bm a:\bm B=0\}$ and $\eta+\kappa\abs{\bm a}^2$ on $\mathrm{span}\{\bm a\}$.
In particular,
\[
  \min\{\eta,\eta+\kappa\abs{\bm a}^2\}\,\abs{\bm B}^2
  \le (\mathcal T(\bm B):\bm B)\le
  \max\{\eta,\eta+\kappa\abs{\bm a}^2\}\,\abs{\bm B}^2,
\]
and $\mathcal T\succ 0$ if and only if both eigenvalues are positive, i.e.\
$\eta>0$ and $\eta+\kappa\abs{\bm a}^2>0$.
\end{lemma}

\begin{proof}
Decompose $\bm B=\bm B_\perp+\alpha\bm a$ with $\bm a:\bm B_\perp=0$ and
$\alpha=(\bm a:\bm B)/\abs{\bm a}^2$.
Then $\mathcal T(\bm B_\perp)=\eta \bm B_\perp$ and
$\mathcal T(\bm a)=(\eta+\kappa\abs{\bm a}^2)\bm a$, giving the spectrum and the bounds.
\end{proof}

Applying Lemma~\ref{lem:rankone-tangent} to \eqref{eq:app_exact_tangent} with
$\eta=\etaeff{\bm A}$, $\kappa=\beta(\bm A)$, and $\bm a=\bm A$, we obtain:
if $\bm A=\bm 0$, then $\DStress{\bm 0}{\bm H}=\etaeff{\bm 0}\bm H$; if $\bm A\neq \bm 0$,
the tangent has exactly two eigenvalues,
\begin{equation}\label{eq:app_eigs}
  \lambda_\perp(\bm A)=\etaeff{\bm A},
  \qquad
  \lambda_\parallel(\bm A)=\etaeff{\bm A}+\beta(\bm A)\abs{\bm A}^2,
\end{equation}
with $\lambda_\perp(\bm A)$ on $\{\bm H:\bm A:\bm H=0\}$ (multiplicity $\dim(\mathbb{R}^{d\times d}_{\sym})-1$)
and $\lambda_\parallel(\bm A)$ in the direction $\bm A$.
Moreover, since $\etaeff{\bm A}>0$ and $\reg^2+(\pn-1)\abs{\bm A}^2>0$ for $\pn>1$, we have
$\lambda_\parallel(\bm A)>0$, hence $\DStress{\bm A}{\cdot}$ is symmetric
positive definite for all $\bm A$.
Then, with $r\coloneq \abs{\bm A}$, \eqref{eq:app_eigs} yields
\[
  \lambda_\parallel(\bm A)
  =\nu_{\infty}+\nu(\reg^2+r^2)^{\frac{\pn-4}{2}}\bigl(\reg^2+(\pn-1)r^2\bigr).
\]
For $1<\pn < 2$ we have $\beta(\bm A)\le 0$, hence $\lambda_\parallel(\bm A)\le \lambda_\perp(\bm A)$ and
thus the smallest eigenvalue acts in the stress direction $\bm A$.
\begin{equation}\label{eq:app_ratio}
  \frac{\lambda_\perp(\bm A)}{\lambda_\parallel(\bm A)}
  =\frac{\nu_{\infty}+\nu(\reg^2+r^2)^{\frac{\pn-2}{2}}}
         {\nu_{\infty}+\nu(\reg^2+r^2)^{\frac{\pn-4}{2}}\bigl(\reg^2+(\pn-1)r^2\bigr)}.
\end{equation}
In particular, if $\nu_{\infty}=0$, then the exact constitutive tangent becomes increasingly anisotropic as $\pn\downarrow 1$ in the
high-shear regime $r\gg \reg$:
\begin{equation}\label{eq:app_ratio_nu0}
  \frac{\lambda_\perp(\bm A)}{\lambda_\parallel(\bm A)}
  =\frac{\reg^2+r^2}{\reg^2+(\pn-1)r^2}
  \le \frac{1}{\pn-1}\,.
\end{equation}

\section{Characterization of local Vanka patch surrogate}
\label{app:patch-surrogate}
To justify single-time-point coefficient surrogates in the local Vanka solves,
we compare the surrogate patch matrix with the exact patch matrix on each slab.
In this appendix we continue the finest-level notation from
\Cref{sec:mg-framework} on a fixed slab $\Sn$,
writing $\tilde{\bm A}_{L,K}^n$ and $\bm A_{L,K}^n$.

\begin{lemma}[Surrogate perturbation for finest-level Vanka patches]
\label{lem:patch-surrogate}
Fix a slab $\Sn$, a patch $K\in\mathcal M_L$, and set
$N_K\coloneq M_K^{\boldsymbol v}+M_K^p$. Let
\[
  \mathfrak A_{n,K}:
  \bigl(\mathbb R^{M_K^{\boldsymbol v}}\bigr)^{k_L+1}
  \to
  \mathbb R^{(k_L+1)N_K\times (k_L+1)N_K}
\]
denote the exact finest-level patch matrix as a function of the local velocity
coefficient vectors at the Gauss-Radau points
$t_n^1,\dots,t_n^{k_L+1}\subset I_n$. Define
\[
  \bm A_{L,K}^n
  \coloneq
  \mathfrak A_{n,K}(\bm U_{n,K}^1,\dots,\bm U_{n,K}^{k_L+1}),
  \qquad
  \tilde{\bm A}_{L,K}^n
  \coloneq
  \mathfrak A_{n,K}(\bm U_{n,K}^{\mathrm{rep}},\dots,\bm U_{n,K}^{\mathrm{rep}}),
\]
where $\bm U_{n,K}^\mu\in\mathbb R^{M_K^{\boldsymbol v}}$ is the local velocity
coefficient vector at $t_n^\mu$, and $\bm U_{n,K}^{\mathrm{rep}}$ corresponds
to a representative time $t_n^{\mathrm{rep}}\in I_n$. All vector norms are
Euclidean, all matrix norms are the induced operator \(2\)-norms, and
\[
  \mathcal W(\bm B)
  \coloneq
  \{\langle \bm x,\bm B\bm x\rangle:
  \bm x\in\mathbb C^{(k_L+1)N_K},\ \|\bm x\|=1\}.
\]

Assume that $\bm A_{L,K}^n$ is invertible with
$\|(\bm A_{L,K}^n)^{-1}\|\le M_{n,K}$, and that:

\begin{enumerate}[label=(\roman*),leftmargin=*,itemsep=2pt,topsep=2pt,parsep=0pt]
\item there exist a bounded set
$\mathcal U_{n,K}\subset\mathbb R^{M_K^{\boldsymbol v}}$ containing all
$\bm U_{n,K}^\mu$ and $\bm U_{n,K}^{\mathrm{rep}}$, and a constant $L_{n,K}>0$,
such that
\begin{equation}\label{eq:patch-lipschitz}
  \|\mathfrak A_{n,K}(\bm V^1,\dots,\bm V^{k_L+1})
   -\mathfrak A_{n,K}(\bm W^1,\dots,\bm W^{k_L+1})\|
  \le
  L_{n,K}\sum_{\mu=1}^{k_L+1}\|\bm V^\mu-\bm W^\mu\|
\end{equation}
for all $(\bm V^\bullet),(\bm W^\bullet)\in(\mathcal U_{n,K})^{k_L+1}$;

\item there exists \(C_K>0\), independent of \(\tau_n\), such that
\begin{equation}\label{eq:patch-approx}
  \|\bm U_{n,K}^\mu-\bm U_{n,K}^{\mathrm{rep}}\|
  \le C_K\tau_n,
  \qquad \mu=1,\dots,k_L+1.
\end{equation}
\end{enumerate}

With
\[
  \bm E_{n,K}\coloneq \tilde{\bm A}_{L,K}^n-\bm A_{L,K}^n,
  \qquad
  \varepsilon_{n,K}\coloneq
  \|(\bm A_{L,K}^n)^{-1}\bm E_{n,K}\|,
\]
the following hold:

\begin{enumerate}[label=(\alph*),leftmargin=*,itemsep=2pt,topsep=2pt,parsep=0pt]
\item \emph{Perturbation bound.}
\begin{equation}\label{eq:patch-perturbation-bound}
  \|\bm E_{n,K}\|
  \le (k_L+1)L_{n,K}C_K\tau_n,
  \qquad
  \varepsilon_{n,K}
  \le M_{n,K}(k_L+1)L_{n,K}C_K\tau_n.
\end{equation}
Hence \(\varepsilon_{n,K}<1\) for all sufficiently small \(\tau_n\).

\item \emph{Preconditioned operator.}
\begin{equation}\label{eq:patch-spectrum-fov}
  \sigma\!\bigl((\bm A_{L,K}^n)^{-1}\tilde{\bm A}_{L,K}^n\bigr),\;
  \mathcal W\!\bigl((\bm A_{L,K}^n)^{-1}\tilde{\bm A}_{L,K}^n\bigr)
  \subset
  \{z\in\mathbb C:\ |z-1|\le \varepsilon_{n,K}\}.
\end{equation}

\item \emph{Surrogate stability.}
If \(\varepsilon_{n,K}<1\), then \(\tilde{\bm A}_{L,K}^n\) is invertible,
\begin{equation}\label{eq:patch-inverse-bound}
  \|(\tilde{\bm A}_{L,K}^n)^{-1}\|
  \le \frac{M_{n,K}}{1-\varepsilon_{n,K}},
\end{equation}
and
\begin{equation}\label{eq:patch-sv-bounds}
  (1-\varepsilon_{n,K})\,s_{\min}(\bm A_{L,K}^n)
  \le s_{\min}(\tilde{\bm A}_{L,K}^n)
  \le s_{\max}(\tilde{\bm A}_{L,K}^n)
  \le (1+\varepsilon_{n,K})\,s_{\max}(\bm A_{L,K}^n).
\end{equation}
\end{enumerate}
\end{lemma}

\begin{proof}
By definition and \eqref{eq:patch-lipschitz},
\[
  \|\bm E_{n,K}\|
  =
  \|\mathfrak A_{n,K}(\bm U_{n,K}^{\mathrm{rep}},\dots,\bm U_{n,K}^{\mathrm{rep}})
   -\mathfrak A_{n,K}(\bm U_{n,K}^{1},\dots,\bm U_{n,K}^{k_L+1})\|
  \le
  L_{n,K}\sum_{\mu=1}^{k_L+1}
  \|\bm U_{n,K}^\mu-\bm U_{n,K}^{\mathrm{rep}}\|.
\]
Using \eqref{eq:patch-approx} gives \eqref{eq:patch-perturbation-bound}, and
hence
\[
  \varepsilon_{n,K}
  \le \|(\bm A_{L,K}^n)^{-1}\|\,\|\bm E_{n,K}\|
  \le M_{n,K}(k_L+1)L_{n,K}C_K\tau_n.
\]

For \eqref{eq:patch-spectrum-fov}, write
\[
  (\bm A_{L,K}^n)^{-1}\tilde{\bm A}_{L,K}^n
  = \bm I+\bm F,
  \qquad
  \bm F\coloneq (\bm A_{L,K}^n)^{-1}\bm E_{n,K},
\]
so that \(\|\bm F\|=\varepsilon_{n,K}\). If
\(\lambda\in \sigma(\bm I+\bm F)\) and \(\|\bm x\|=1\) is a corresponding
eigenvector, then \(\bm F\bm x=(\lambda-1)\bm x\), hence
\(|\lambda-1|\le \|\bm F\|=\varepsilon_{n,K}\). Likewise, for
\(\|\bm x\|=1\),
\[
  |\langle \bm x,(\bm I+\bm F)\bm x\rangle-1|
  = |\langle \bm x,\bm F\bm x\rangle|
  \le \|\bm F\|
  = \varepsilon_{n,K},
\]
which proves \eqref{eq:patch-spectrum-fov}.

If \(\varepsilon_{n,K}<1\), then
\[
  \tilde{\bm A}_{L,K}^n
  = \bm A_{L,K}^n(\bm I+\bm F),
\]
and the Neumann series yields
\[
  \|(\bm I+\bm F)^{-1}\|\le (1-\varepsilon_{n,K})^{-1}.
\]
Therefore,
\[
  \|(\tilde{\bm A}_{L,K}^n)^{-1}\|
  \le
  \frac{\|(\bm A_{L,K}^n)^{-1}\|}{1-\varepsilon_{n,K}}
  \le
  \frac{M_{n,K}}{1-\varepsilon_{n,K}},
\]
which proves \eqref{eq:patch-inverse-bound}. Finally, the product inequalities
for singular values and
\[
  (1-\|\bm F\|)\|\bm x\|
  \le \|(\bm I+\bm F)\bm x\|
  \le (1+\|\bm F\|)\|\bm x\|
\]
imply
\[
  1-\varepsilon_{n,K}\le s_{\min}(\bm I+\bm F)
  \le s_{\max}(\bm I+\bm F)\le 1+\varepsilon_{n,K},
\]
and hence \eqref{eq:patch-sv-bounds}.
\end{proof}

\begin{remark}[Scope of the assumptions]
\label{rem:patch-surrogate}
The essential structural hypothesis is the invertibility of the realized exact
patch matrix \(\bm A_{L,K}^n\) with bound \(M_{n,K}\); the result is therefore
step-local rather than neighborhood-uniform. The localized Lipschitz condition
\eqref{eq:patch-lipschitz} is natural for the present
\((\pn,\reg)\)-Navier-Stokes setting, since \(\reg>0\) yields smooth dependence
of the coefficient-dependent patch contributions on bounded sets of local
velocity states. The approximation property \eqref{eq:patch-approx} expresses
\(O(\tau_n)\) variation of the local velocity coefficients across the slab and
is used here as an assumption. The inclusion \eqref{eq:patch-spectrum-fov} shows
that the preconditioned surrogate patch operator remains close to the identity
in spectrum and numerical range, which is the relevant local approximation
property for the surrogate Vanka solve.
\end{remark}

\bibliographystyle{elsarticle-num}
\bibliography{references_pstokes}

\end{document}